\documentclass[12pt]{amsart}

\usepackage{latexsym}
\usepackage{amsmath}
\usepackage{amssymb}
\usepackage{amscd}
\usepackage{multicol}
\usepackage[cmtip,matrix,arrow]{xy}
\usepackage[hypertex]{hyperref} 
\numberwithin{equation}{section}

\newtheorem{theorem}{\bf Theorem}[section]
\newtheorem{corollary}[theorem]{\bf Corollary}
\newtheorem{proposition}[theorem]{\bf Proposition}
\newtheorem{lemma}[theorem]{\bf Lemma}
\newtheorem{definition}[theorem]{\bf Definition}
\newtheorem{example}[theorem]{\bf Example}
\newtheorem{definition-theorem}[theorem]{\bf Theorem-Definition}
\newtheorem{remark}[theorem]{\bf Remark}

\def\bR{\mathbb{R}}
\def\bC{\mathbb{C}}
\def\bZ{\mathbb{Z}}
\def\bS{\mathbb{S}}

\def\bH{\mathbb{H}}
\def\K{\mathbb{K}}

\def\g{\mathfrak{g}}
 
  \def\c{\mathfrak{c}}
\def\quott({/\! /}

\def\g{\mathfrak{g}}

\def\p{\mathfrak{p}}

\def\u{\mathfrak{u}}

\def\o{\mathfrak{o}}

\def\A{{\mathcal A}}
\def\B{{\mathcal B}}

\def\SO{{\rm SO}}
\def\O{{\rm O}}
\def\SU{{\rm  SU}}
\def\S{{\rm  S}}
\def\U{{\rm  U}}
\def\G{{\rm  G}}
\def\Sp{{\rm  Sp}}

\def\tP{{P}}

\def\tQ{\tilde{Q}}

\def\bt{\bar{\tau}}
\title[Bott periodicity for inclusions]
{ Bott periodicity for inclusions of  symmetric spaces}
\author[A.-L. Mare]{Augustin-Liviu\ Mare}
\author[P. Quast]{Peter \ Quast}

\date{}

\begin{document}
\address[A.-L. Mare]{Department of Mathematics and Statistics\\ University of Regina\\ Canada}
\address[P. Quast]{Institut f\"{u}r Mathematik\\ Universit\"{a}t Augsburg\\ Germany}

\email[A.-L. Mare]{mareal@math.uregina.ca}
\email[P. Quast]{peter.quast@math.uni-augsburg.de}

\begin{abstract}
When looking at Bott's original proof of his periodicity theorem for the stable homotopy groups of the orthogonal and unitary groups, one sees in the background  a differential geometric periodicity phenomenon. We show that this geometric  phenomenon extends to the standard inclusion of the orthogonal group  into the unitary group. 
Standard inclusions between other
classical Riemannian symmetric spaces are considered as well. An application to homotopy theory is also discussed. 

\vspace{0.3cm}


\end{abstract}

\maketitle

\tableofcontents

\newpage

\section{Introduction}

Bott's original proof of his periodicity theorem  \cite{Bo} 
is   differential geometric in its nature.
It relies on the observation that in  a compact Riemannian symmetric space $P$ one can choose two points $p$ and $q$ in ``special position"  such that the connected components of the space
of shortest geodesics in $P$ joining $p$ and $q$
are again compact symmetric spaces.
Set $P_0=P$ and let $P_1$ be one of the resulting
 connected components.
 This construction can be repeated inductively:
 given points $p_j, q_j$ in ``special position" in $P_j$, then
 $P_{j+1}$ is one of the connected components of the space of shortest geodesic segments
 in $P_j$ between $p_j$ and $q_j$.
 If we start this iterative process with the classical groups
 $$ P_0:=\SO_{16n},\quad  \tilde{P}_0:=\U_{16n},\quad \bar{P}_0:=\Sp_{16n}$$
and make at each step appropriate choices of the two points and of the connected component, one obtains
$$P_8=\SO_n,\quad \tilde{P}_2 = \U_{8n},\quad \bar{P}_8=\Sp_n.$$
Each of the three processes can be continued, provided that $n$ is divisible
by a sufficiently high power of $2$. We obtain (periodically)
copies of a special orthogonal, unitary, and symplectic group after
every eighth, second, respectively eighth iteration.
These purely geometric periodicity phenomena are the key ingredients of
Bott's proof of his periodicity theorems \cite{Bo} for the stable homotopy groups
$\pi_i(\O)$, $\pi_i(\U)$, and $\pi_i(\Sp)$ (see also the remark at the end of this section).

In his book \cite{Mi}, Milnor constructed  totally geodesic embeddings
$$P_{k+1}\subset P_k, \quad \tilde{P}_{k+1}\subset \tilde{P}_k, \quad
\bar{P}_{k+1}\subset \bar{P}_k, $$ for all $k=0,1,\ldots, 7$.
In each case, the inclusion is given by the map which  assigns to a geodesic
its midpoint (cf.~\cite{Qu} and \cite{Ma-Qu}, see also Section \ref{sectiondoi} below).

The goal of this paper is to establish connections between the following three
chains of symmetric spaces:
\begin{align*}
P_0\supset P_1\supset P_2 \supset \ldots \supset P_8,\\
\tilde{P}_0\supset \tilde{P}_1\supset \tilde{P}_2 \supset \ldots \supset \tilde{P}_8,\\
\bar{P}_0\supset \bar{P}_1\supset \bar{P}_2 \supset \ldots \supset \bar{P}_8.
\end{align*}
We will refer to them as the $\SO$-, $\U$-, respectively $\Sp$-{\it Bott chains.}
Starting with the natural inclusions
$$ P_0=\SO_{16n}\subset \U_{16n} =\tilde{P_0}
 \quad {\rm and} \quad
 \tilde{P}_0 = \U_{16n} \subset \Sp_{16n}=\bar{P}_0$$
 we show that the iterative process above provides inclusions
 $$P_j\subset \tilde{P}_j \quad {\rm and} \quad
 \tilde{P}_j\subset \bar{P}_j$$
 for all $j=0,1,\ldots, 8.$ These are all canonical reflective inclusions of symmetric spaces,
 i.~e.~they can be realized as  fixed point sets of isometric involutions  (see Appendix \ref{apa},
 especially Tables 5 and 6 and Subsections \ref{1e} - \ref{8em})
 and make the following diagrams commutative:

  \begin{equation*}\label{ex}
\vcenter{\xymatrix{
P_0\ar[d]^{\cap} &
P_1\ar[d]^{\cap}
\ar[l]_{  \supset  } &
P_2
\ar[l]_{\supset} \ar[d]^{\cap}  & \ar[l]_{\supset}\cdots&
P_8\ar[l]_{\supset}\ar[d]^{\cap} &
\\
\tilde{P}_0 &
\tilde{P}_1
\ar[l]_{  \supset  } &
\tilde{P}_2
\ar[l]_{\supset}  & \ar[l]_{\supset}\cdots&
\tilde{P}_8\ar[l]_{\supset} &\\
}}
\end{equation*}

 \begin{equation*}\label{pseq}
\vcenter{\xymatrix{
\tilde{P}_0\ar[d]^{\cap} &
\tilde{P}_1\ar[d]^{\cap}
\ar[l]_{  \supset  } &
\tilde{P}_2
\ar[l]_{\supset} \ar[d]^{\cap}  & \ar[l]_{\supset}\cdots&
\tilde{P}_8\ar[l]_{\supset}\ar[d]^{\cap} &
\\
\bar{P}_0 &
\bar{P}_1
\ar[l]_{  \supset  } &
\bar{P}_2
\ar[l]_{\supset}  & \ar[l]_{\supset}\cdots&
\bar{P}_8\ar[l]_{\supset} &\\
}}
\end{equation*}
Moreover, the vertical inclusions are periodic,  with period equal to $8$.
Concretely, we show that up to isometries, the inclusions
$$P_8\subset \tilde{P}_8 \quad {\rm and} \quad \tilde{P}_8\subset \bar{P}_8$$ are again the natural inclusions
 $$\SO_n\subset \U_n \quad {\rm and} \quad \U_n\subset \Sp_n$$
(see Theorems \ref{co}, \ref{coc} and Remark \ref{ream} below).
We mention that all inclusions in the two diagrams above are actually reflective. 
For example, notice that $P_4=\Sp_{2n}$, $\tilde{P}_4=\U_{4n}$,
 and $\bar{P}_4=\SO_{8n}$; the inclusions
 $$P_4\subset \tilde{P}_4 \quad {\rm and } \quad \tilde{P}_4\subset \bar{P}_4$$
 are essentially the usual subgroup inclusions
 $$\Sp_{2n}\subset \U_{4n} \quad {\rm and} \quad \U_{4n}\subset \SO_{8n}$$
 (see Remark \ref{samesp}).

\noindent{\bf Remark.}
We recall that the celebrated periodicity theorem of Bott \cite{Bo} concerns the
stable homotopy groups $\pi_i(\O)$, $\pi_i(\U)$, and $\pi_i(\Sp)$ of the
orthogonal, unitary, respectively symplectic groups.
Concretely, one has the following group isomorphisms:
$$\pi_i(\O)\simeq \pi_{i+8}(\O),\quad \pi_i(\U)\simeq \pi_{i+2}(\U),
\quad   \pi_{i}(\Sp)\simeq \pi_{i+8}(\Sp),$$
for all $i\ge 0$.
If we now consider the standard inclusions 
\begin{equation}\label{onun}\O_n\subset \U_n \quad {\rm and} \quad \U_n\subset \Sp_n
\end{equation}
then  the maps induced between homotopy groups, that is
$\pi_i(\O_n)\to \pi_i(\U_n)$  and $\pi_i(\U_n)\to \pi_i(\Sp_n)$
 are stable
relative to $n$ within the ``stability range".
One can see that the resulting maps
$$f_i:\pi_i(\O)\to \pi_i(\U) \quad {\rm and} \quad g_i : \pi_i(\U)\to \pi_i(\Sp),$$
are periodic in the following sense:
\begin{equation}\label{fi8}f_{i+8}=f_i \quad {\rm and} \quad g_{i+8}=g_i.\end{equation}
These facts are basic in  homotopy theory and can be proved using techniques 
described e.g.~in \cite[Ch.~1]{May}.
We   provide an alternative, more elementary
proof of Equation (\ref{fi8}) and  determine  the maps $f_i$ and $g_i$ explicitely, by using  only on the long exact homotopy sequence of the principal bundles
$\O_n\to \U_n\to \U_n/\O_n$ and $\U_n\to \Sp_n\to \Sp_n/\U_n$,
combined with the explicit knowledge of the stable
homotopy groups of $\O$, $\U$, $\Sp$, $\U/\O$, and $\Sp/\U$
(the details can be found in Section \ref{secper}, see especially Theorems
\ref{maintheo} and \ref{maitheo}, Remarks \ref{notes} and \ref{remer},
and Tables 1 and 3).
The present paper shows that  the results stated by Equation (\ref{fi8}) are just direct  consequences of 
the abovementioned differential geometric periodicity phenomenon, in the spirit of
Bott's original proof of his periodicity theorems.
Besides the inclusions given by Equation (\ref{onun}) we will also consider the following
ones, which are described in detail in Appendix \ref{apa},
Subsections \ref{1e} - \ref{8em}:
\begin{align*}
{} & \O_{2n}/\U_n \subset \G_{n}(\bC^{2n}), & {} &
\U_{2n}/\Sp_n\subset \U_{2n}, & {} &
\G_{n}(\bH^{2n}) \subset \G_{2n}(\bC^{4n}), \\
{} &  \Sp_n \subset \U_{2n},& {} &
\Sp_n/\U_n\subset \G_{n}(\bC^{2n}), & {} &
\U_n/\O_n\subset \U_n, \\
{} & \G_{n}(\bR^{2n}) \subset \G_{n}(\bC^{2n}),& {} &
\G_{n}(\bC^{2n})\subset \Sp_{2n}/\U_n,  & {} &
\U_n\subset \U_{2n}/\O_{2n},  \\
{} & \G_{n}(\bC^{2n}) \subset \G_{2n}(\bR^{4n}), & {} &
\U_n \subset \O_{2n}, & {} &
\G_{n}(\bC^{2n}) \subset \O_{4n}/\U_{2n}, \\
{} & \U_n\subset \U_{2n}/\Sp_n,   & {} &
\G_{n}(\bC^{2n}) \subset \G_{n}(\bH^{2n}). & {} & {}
\end{align*}
For each of them we will prove a periodicity result similar to those described by Equation
(\ref{onun}).
The precise statements are Corollaries \ref{cor} and \ref{coragain}.

 \vspace{0.3cm}

\noindent {\bf Acknowledgements.} We would like to thank Jost-Hinrich Eschenburg for discussions about the topics of the paper. We are also grateful to the Mathematical Institute at the University of Freiburg, especially Professor Victor Bangert, for  hospitality  while part of this work was being done. The second named author
wishes to thank the University of Regina for hosting him during a research visit
in March 2010.

\section{Bott periodicity from a geometric viewpoint}\label{sectiondoi}

In this section we  review the original (geometric) proof of Bott's periodicity theorem.
We adapt the original treatment in \cite{Bo} to our needs and therefore
change it slightly. More precisely,
we will use
ideas of Milnor \cite{Mi},
as well as the concept of centriole, which was defined by Chen and Nagano
\cite{Ch-Na1} (see also \cite{Na}, \cite{Na-Su}, and \cite{Bu}).

 \subsection{The geometry of centrioles.}
Let $P$ be a compact connected symmetric space and $o$ a point in $P$.
We say that $(P,o)$ is a {\it pointed symmetric space}.
As already mentioned in the introduction, a key role is played by the
space of all shortest geodesic segments in $P$ from $o$ to a
point in $P$ which belongs to a certain ``special" class.
It turns out that this class consists of the poles of $(P,o)$, 
(cf.~\cite{Qu} and \cite{Ma-Qu}). The notion of pole is described by the following definition.
First, for any $p\in P$ we denote by $s_p:P\to P$ the corresponding geodesic symmetry. 

\begin{definition}\label{ole}
A {\rm pole} of the pointed symmetric space $(P,o)$ is a point $p\in P$ with the property that
$s_p=s_o$ and $p\neq o$.
\end{definition}

Let $G$ be the identity component of the isometry group of $P$.
This group acts transitively on $P$. We denote by $K$ the $G$-stabilizer
of $o$ and by $K_e$ its identity component.
The following result is related to \cite[Vol.~II, Ch.~VI, Proposition 2.1 (b)]{Lo}.

\begin{lemma}\label{ele}
If $p$ is a pole of $(P,o)$, then
$k.p=p$ for all $k\in K_e$.
\end{lemma}

\begin{proof}  
The map $\sigma : G\to G$, $\sigma(g) =s_ogs_o$ is an involutive group automorphism of $G$ whose fixed point set  $G^\sigma$ has the same identity
component  $K_e$ as $K$.
Since $p$ is a pole, we have $\sigma(g)=s_pgs_p$ and the fixed point set
$G^\sigma$ has the same identity component as the stabilizer $G_p$ of $p$ in $G$.
Consequently, $K_e \subset G_p$.
\end{proof}

\begin{example}\label{anyc} {\rm Any compact connected Lie group $G$ can be equipped with a
bi-invariant metric and becomes in this way a Riemannian symmetric space
(cf.~e.g.~\cite[Section 21]{Mi}). The geodesic symmetry at $g\in G$ is the map
$s_g:G\to G$, $s_g(x)=gx^{-1}g$, $x\in G$.
An immediate consequence is a description of the poles of $G$: they are exactly those $g$ which lie in the center of $G$ and whose square is equal to the identity of $G$. We also note that the identity component of the isometry group of $G$ is $G\times G/\Delta(Z(G))$,
where $\Delta(Z(G)):=\{(z,z) \ : \ z\in Z(G)\}$. Here $G\times G$ acts on $G$ via
\begin{equation}\label{ligr}(g_1,g_2).h:=g_1hg_2^{-1} \quad g_1,g_2,h\in G
\end{equation} and the kernel of this action is equal to
$\Delta(Z(G))$. Finally, the stabilizer of the identity element $e$ of $G$ is $\Delta(G)/\Delta(Z(G))$.}
\end{example}

\begin{remark}\label{notan} {\rm Not any pointed compact symmetric space admits a pole. For example, consider the Grassmannian $\G_k(\K^{2m})$, where
$0 \le k \le 2m$
and $\K=\bR, \bC$, or $\bH$. It has a canonical structure of a Riemannian symmetric space. One can show that $\G_k(\K^{2m})$ has a pole if and only if $k=m$.
Indeed, let us first consider  an element $V$ of $\G_m(\K^{2m})$.
Then a pole of the pointed symmetric space $(\G_m(\K^{2m}), V)$ is
$V^\perp$, the orthogonal complement of $V$ in $\K^{2m}$.

\vspace{-0.2cm}

\noindent From now on, we will assume that $k\neq m$. We  take into account the general fact that  if a compact symmetric space $P$ has a pole, then there is a non-trivial Riemannian double covering $P\to P'$ (see e.g.~\cite[Proposition 2.9]{Ch-Na1} or \cite[Lemma 2.15]{Qu}).
 Now, none of the spaces $\G_k(\K^{2m})$ is a covering of another space,
 in other words, all $\G_k(\K^{2m})$ are adjoint symmetric spaces.
To prove this, we need to consider the following two situations.
If $\K=\bR$, we note that  the symmetric space $\G_k(\bR^{2m})$ has the Dynkin diagram  of type ${\mathfrak b}$, hence it has exactly one simple root with coefficient equal to 1
in the expansion of the highest root
(see \cite{He}, Table V, p.~518, Table IV, p.~532 and the table on p.~477).
On the other hand, $\G_k(\bR^{2m})$  is covered by the  Grassmannian
of all oriented $k$-subspaces in $\bR^{2m}$.
By using the theorem  of Takeuchi \cite{Ta}, the latter space is simply connected,
 and $\G_k(\bR^{2m})$ is its adjoint symmetric space.
 If $\K=\bC$ or $\K=\bH$, we note that
the symmetric space $\G_k(\K^{2m})$ has Dynkin diagram of type
$\mathfrak{bc}$; by using again \cite{Ta}, we deduce that $\G_k(\K^{2m})$ is at the same time simply connected and an adjoint symmetric space.
 }
 \end{remark}

Recall that spaces of shortest geodesic segments with prescribed endpoints in a symmetric space  are an important tool in Bott's proof of his periodicity theorem \cite{Bo}. We can identify such  spaces
with submanifolds by mapping a shortest geodesic segment to its midpoint.
We therefore have a closer look at these spaces. The objects described in the following definition are slightly more general,  in the sense that the geodesic segments are not required to be shortest (we will return to this assumption at the end of this subsection).

\begin{definition}\label{pppole} Let $p$ be a pole of $(P,o)$. The set $C_p(P,o)$ of all midpoints of geodesics
in $P$ from $o$ to $p$ is called a {\rm centrosome}.
The connected components of
a centrosome are called {\rm centrioles}.
\end{definition}

For more on these notions we refer to \cite{Ch-Na1}  and \cite{Na}.
The following result is a consequence of \cite[Proposition 2.12 (ii)]{Na}
(see also \cite[Proposition 2.16]{Qu} or \cite[Proposition 2]{Qu1}).

\begin{lemma}\label{tg} Any centriole in a compact symmetric space is a reflective,
hence totally geodesic  submanifold.  \end{lemma}

We recall that a submanifold of a Riemannian manifold is called {\it reflective} if it is a connected component of the fixed point set of an isometric involution.
Reflective submanifolds of irreducible simply connected Riemannian symmetric spaces have been classified by  Leung in \cite{Le} and \cite{Le1} .
This classification in the special case when the symmetric space is a compact simple Lie group will be an important tool for us (see Appendix \ref{ligrp}).

Although the following result appears to be known (see \cite{Ch-Na1} and \cite{Na}),
we  decided to include a proof of it, for the sake of completeness.

\begin{lemma}\label{pol} Let $p$ be a pole of $(P,o)$. The centrioles of  $(P,o)$ relative to $p$ are orbits of the canonical $K_e$-action on $P$.
\end{lemma}

\begin{proof} Let $C$ be a connected component of $C_p(P, o)$
and take $x\in C$. There exists a geodesic $\gamma: \bR \to P$ such that
$\gamma(0)=o,\gamma(1)=x$, and $\gamma(2)=p$.
For any $k\in K_e$, the restriction of the map $k.\gamma : \bR \to P$
to the interval $[0,2]$ is a  geodesic segment between $o$ and $p$
(see  Lemma  \ref{ele}). Thus
the point $k.\gamma(1)$ is in $C_p(p,o)$. Since $K_e$ is connected, we deduce
that $K_e.x\subset C$.

Let us now prove the converse inclusion. Take $y\in C$ and consider a geodesic
$\mu : \bR \to C$ such that $\mu(0)=x$ and $\mu(1)=y$.
By Lemma \ref{tg}, $\mu$ is a geodesic in $P$ as well.
We consider the one-parameter subgroup of transvections along $\mu$ which is given by
$\tau_\mu:\bR\to G$,
$\tau_\mu(t):=s_{\mu(t/2)}\circ s_{\mu(0)}$
(see e.g.~\cite[Lemma 6.2]{Sa}).

\noindent{\it Claim.} $\tau_\mu(t)\in K_e$, for all $t\in \bR$.

Indeed, since  $\mu(0)$ and $\mu(t/2)$ are both midpoints of geodesic segments between $o$ and $p$, we have $s_{\mu(0)}.o=p$ and $s_{\mu(t/2)}.p=o$.
Hence, $\tau_\mu(t).o=o$. We deduce that $\tau_\mu(t)\in K$.
Since $\tau_\mu(0)$ is the identity transformation of $P$, we actually have
$\tau_\mu(t)\in K_e$.

The claim along with the fact that $\mu(0)=x$ implies that $\tau_\mu(1).x= s_{\mu(1/2)}\circ s_{\mu(0)}.x=s_{\mu(1/2)}.x=y$. Thus $y\in K_e.x$.
\end{proof}

From Lemmata \ref{ele} and \ref{pol}, we see that whenever a centriole in $C_p(P,o)$
contains a midpoint of a shortest geodesic segment between $o$ and $p$, then this
centriole consists of midpoints of such shortest geodesic segments only.
Such centrioles are called {\it s-centrioles}. (For further properties of s-centrioles we refer to \cite{Qu1}.)


\subsection{The $\S\O$-Bott chain}\label{sochain}
We outline Milnor's description \cite[Section 24]{Mi} of this chain. 
The  chain starts with $P_0 = \S\O_{16n}$. 
We then consider the space of
all orthogonal complex structures in $\SO_{16n}$, that is,
$$\Omega_1:=\{J\in \SO_{16n} \ : \ J^2=-I\}.$$ This space has two connected components, which are both diffeomorphic to
$\S\O_{16n}/\U_{8n}$.
We pick any of these two components and denote it by  $P_1$.
For $2 \le k\le 7$ we construct the spaces $P_k\subset \S\O_{16n}$ inductively, as follows:
Assume that $P_k$ has been constructed and pick a base-point $J_k\in P_k$.
We define $P_{k+1}$ as the top-dimensional connected component of the space
$$\Omega_{k+1}:=\{J\in P_k \ : \ JJ_k=-J_kJ\}.$$ In this way we  construct $P_2,\ldots, P_7$.
Finally, we pick $J_7\in P_7$ and define $P_8$ as any of the two
connected components of the space
$$\Omega_8:=\{J\in P_7 \ : \ JJ_7=-J_7J\}.$$ (Note the latter space is diffeomorphic to the orthogonal group $\O_n$, thus it has two components that are diffeomorphic). It turns out that $P_1, \ldots, P_8$ are submanifolds of 
$\S\O_{16n}$, whose diffeomorphism types can be described as follows:
$P_0={\rm SO}_{16n}$, $P_1= \S\O_{16n}/\U_{8n}$,
$P_2= \U_{8n}/\Sp_{4n}$,  $P_3= \G_{2n}(\bH^{4n})=\Sp_{4n}/(\Sp_{2n}\times \Sp_{2n})$, $P_4= \Sp_{2n}$, $P_5= \Sp_{2n}/\U_{2n}$,  $P_6=\U_{2n}/\O_{2n}$,  $P_7= \G_{n}(\bR^{2n})=\SO_{2n}/\S(\O_n\times \O_n)$, and
$P_8= \SO_n$. 
The details can be found in \cite[Section 24]{Mi}.

For our future goals it is useful  to have an alternative description of
the $\S\O$-Bott chain. This  is presented by the following two lemmata.

\begin{lemma}\label{0,}  For any $0 \le k \le 7$, the subspace $P_k$ of $\S\O_{16n}$ is invariant under the automorphism of $\S\O_{16n}$ given by $X\mapsto -X$.
\end{lemma}

\begin{proof} 
First, $\Omega_k$ is obviously invariant under $X\mapsto -X$, $X\in \S\O_{16n}$.
The decisive argument is the information provided by the last paragraph on p.~137 in \cite{Mi}: for any $X\in \Omega_k$,  there exists a path in $\Omega_k$ from $X$ to $-X$.
\end{proof}

Let us now equip $\S\O_{16n}$
with the bi-invariant metric induced  by
\begin{equation}\label{metrcsos}\langle X,Y\rangle =- {\rm tr}(XY),\end{equation}
for all $X,Y$ in the Lie algebra $\o_{16n}$ of $\S\O_{16n}$.
Then $P_1, \ldots, P_8$ are totally geodesic submanifolds of
$\S\O_{16n}$ (see \cite[Lemma 24.4]{Mi}). Fix $k\in \{0,1, \ldots, 7\}$ and set $J_0:=I$.
From Example \ref{anyc} we deduce that $-J_k$ is a pole of $(\S\O_{16n},J_k)$.
 By the previous lemma, $-J_k$ lies in  $P_k$ and, since the latter space is totally geodesic in $\S\O_{16n}$, $-J_k$ is a pole of
 $(P_k, J_k)$. The following lemma follows from the Remark on p.~138 in \cite{Mi}.
 
 \begin{lemma}\label{0<} For any $k\in \{0,1,\ldots, 7\}$, the space $P_{k+1}$ is an s-centriole of $(P_k, J_k)$ relative to the
 pole $-J_k$.
 \end{lemma}

\begin{remark}\label{sclaar} {\rm 
As we will show in Proposition \ref{isotype} (b),
$P_8$ is isometric to $\S\O_n$, the latter being equipped with the standard bi-invariant metric multiplied by a certain scalar. Assume that $n$ is an even integer and pick
$J_8\in P_8$. With the method used in the proof of Lemma \ref{0,}  one can show that $-J_8$ is in $P_8$ as well
(indeed by the footnote on p.~142 in \cite{Mi}, there exists an orthogonal complex structure $J\in \SO_{16n}$ which anti-commutes with $J_1, \ldots, J_7$). As in Lemma \ref{0<},
$-J_8$ is  a pole of $(P_8, J_8)$ and,
by using  Example \ref{anyc} for $G=\S\O_n$, it is the only one.
We conclude that the $\S\O$-Bott chain can be extended and is periodic, in the sense that if $n$ is divisible by a ``large" power of $16$, then every eighth element of the chain
is isometric to  a certain special orthogonal group equipped with a bi-invariant metric.}
\end{remark}

\subsection{The $\Sp$-Bott chain}\label{spch}
This is obtained from the $\S\O$-chain  by taking $P_4$ as the
initial element.
More precisely, we replace  $n$ by $8n$ and, in this way, 
 $P_4$ is diffeomorphic to $\Sp_{16n}$. This is the first term of the $\Sp$-chain, call 
 it $\bar{P}_0$. Here is the list of all terms of the chain, described up to
 diffeomorphism:
$\bar{P}_0=\Sp_{16n},
\bar{P}_1=\Sp_{16n}/\U_{16n},
\bar{P}_2=\U_{16n}/\O_{16n},
\bar{P}_3=\G_{8n}(\bR^{16n})=\SO_{16n}/\S(\O_{8n}\times \O_{8n}),
\bar{P}_4=\SO_{8n},
\bar{P}_5=\SO_{8n}/\U_{4n},
\bar{P}_6=\U_{4n}/\Sp_{2n},
\bar{P}_7=\G_{n}(\bH^{2n})=\Sp_{2n}/\Sp_n\times \Sp_n,$ and $\bar{P}_8=\Sp_n$.
As explained in the previous subsection, these are Riemannian manifolds obtained by successive applications of the centriole construction.
The starting point is $P_0=\Sp_{16n}$ with the Riemannian metric which  is described
at the beginning of Section \ref{inclubot}:  by Proposition \ref{isotype} (a),
this metric  is the same as the submanifold metric on $P_4$, up to a scalar multiple.

\subsection{Poles and centrioles in $\U_{2q}$}\label{poled}
Let $q$ be an integer, $q\ge 1$. We equip the unitary group $\U_{2q}$ with the bi-invariant metric induced by the inner product
\begin{equation}\label{lega}\langle X, Y\rangle =-{\rm tr}(XY),\end{equation}
for all $X,Y$ in the Lie algebra $\u_{2q}$ of $\U_{2q}$.
The center of $\U_{2q}$  is
$$Z(\U_{2q})=\{zI \ : \  z\in \bC, |z|=1\}.$$ From Example \ref{anyc}, the  pointed symmetric space
$(\U_{2q}, I)$ has exactly one pole, namely the matrix $-I$.
By  Lemma \ref{pol},
the centrioles of $(\U_{2q}, I)$ are certain orbits of the conjugation action
of $\U_{2q}$ on itself, since they coincide with the orbits of the action
of $\U_{2q}/Z(\U_{2q})$.

Let us describe explicitly the s-centrioles.
We first describe the
 shortest geodesic segments  in $\U_{2q}$ between
$I$ and $-I$, that is, $\gamma:[0,1] \to \U_{2q}$ such that
$\gamma(0)=I$ and $\gamma(1)=-I$.
Any such $\gamma$ is $\U_{2q}$-conjugate to the $1$-parameter subgroup
\begin{equation}\label{gamak}\gamma_k: \ t\mapsto
\exp \left[t
\left(%
\begin{array}{ccccccc}
   \pi i I_k& 0\\
0 & -\pi i I_{2q-k}
\end{array}%
\right)\right],  t\in \bR
\end{equation}
restricted to the interval $[0,1]$, for some $0\le k \le 2q$
(see \cite[Section 23]{Mi}).
Consequently, any s-centriole is of the form $\U_{2q}.\gamma_k\left(\frac{1}{2} \right)$,
that is, the $\U_{2q}$-conjugacy class of
$$\exp \left[\frac{1}{2}
\left(%
\begin{array}{ccccccc}
   \pi i I_k& 0\\
0 & -\pi i I_{2q-k}
\end{array}%
\right)\right]
=
\left(%
\begin{array}{ccccccc}
    i I_k& 0\\
0 & - i I_{2q-k}
\end{array}%
\right).
$$
The $\U_{2q}$-stabilizer of this matrix is $\U_k\times \U_{2q-k}$, hence one
can identify the orbit with $\U_{2q}/\U_k\times \U_{2q-k}$, which is just the Grassmannian $\G_k(\bC^{2q})$.  If we equip the orbit with the submanifold Riemannian metric, then the (transitive) conjugation action of $\U_{2q}$ on it is isometric, in other words, the  metric is
 $\U_{2q}$-invariant.
 Note that up to a scalar there is a unique such metric on $\G_k(\bC^{2q})$ and it makes this space into a symmetric space.

We will be especially interested in the centriole corresponding to $k=q$,
which we call the {\it top-dimensional s-centriole}.
Concretely, this is the $\U_{2q}$-conjugacy class of 
the matrix
 \begin{equation}\label{aq}
A_q:=
 \left(%
\begin{array}{ccccccc}
    i I_q& 0\\
0 &  -i I_{q}
\end{array}%
\right)
\end{equation}
and it is isometric to the Grassmannian $\G_q(\bC^{2q})$ equipped with a canonical
symmetric space metric.

Finally, note that if instead of $I$ the base point is an arbitrary element $A$ of $\U_{2q}$, then
the only pole of $(\U_{2q},A)$ is the matrix $-A$.
The corresponding centrioles are
$A(\U_{2q}.\gamma_k\left(\frac{1}{2}\right))$, that is,
$A$-left translates in $\U_{2q}$ of the
conjugacy classes described above. As before, they are all s-centrioles.

\begin{remark}\label{remra} {\rm The top-dimensional s-centriole of
$(\U_{2q}, A)$ relative to $-A$ is invariant under the automorphism of $\U_{2q}$
given by $X\mapsto -X$. The reason is that the matrix $-A_q$ is $\U_{2q}$-conjugate 
to $A_q$.}
\end{remark}


\subsection{Poles and centrioles in $\G_q(\bC^{2q})$}\label{poleun} 
We regard the Grassmannian  $\G_q(\bC^{2q})$ as the top-dimensional  s-centriole
of $(\U_{2q}, I)$ relative to $-I$, that is,  the conjugacy class in $\U_{2q}$  of
 the matrix
 $
A_q$ described by Equation (\ref{aq}).
Note that,
by Remark \ref{remra}, if $A$ is in $\G_q(\bC^{2q})$, then $-A$ is in
$\G_q(\bC^{2q})$, too.



\begin{lemma}\label{gonepole}
If $A\in \G_q(\bC^{2q})$, then the pointed symmetric space
$(\G_q(\bC^{2q}), A)$ has only one pole, which is $-A$.
\end{lemma}
\begin{proof}
First, observe that the geodesic symmetries $s_A$ and $s_{-A}$ of $\U_{2q}$ are
identically equal (see Example \ref{anyc}).
By Lemma \ref{tg}, $\G_q(\bC^{2q})$ is a totally geodesic submanifold of
$\U_{2q}$. Hence, $-A$ is a pole of $(\G_q(\bC^{2q}), A)$.  We claim that the pointed symmetric space $(\G_q(\bC^{2q}), A)$ has at most one pole.
Indeed, let $\pi$ be the Cartan map of $\G_q(\bC^{2q})$, i.e.~the map
that assigns to each point its geodesic symmetry. It is known that this is
a Riemannian covering onto its image, the latter being a compact symmetric space.
Observe that the fundamental group of the adjoint space of $\G_q(\bC^{2q})$ is $\bZ_2$. We prove this by using the same kind of  argument as in the second half of Remark \ref{notan}: the Dynkin diagram of the symmetric space $\G_q(\bC^{2q})$ is of
type ${\mathfrak c}$, hence there is exactly one simple root with coefficient equal to
1 in the expansion of the highest root
(see \cite{He}, Table V, p.~518, Table IV, p.~532 and the table on p.~477); we use
again the theorem of Takeuchi \cite{Ta}.
Since
 $\G_q(\bC^{2q})$ is simply connected and we have $\pi(A)=\pi(-A)$ we deduce that
 $\pi$ is a double covering. Finally, we take into account that any pole of
$(\G_q(\bC^{2q}), A)$
is in the pre-image $\pi^{-1}(\pi(A))$.
 \end{proof}

We note that this lemma is related to \cite[Proposition 2.23 (i)]{Na2}.

\begin{remark} {\rm Recall that, by definition, $\G_q(\bC^{2q})$ is the space of all
$q$-dimensional complex vector subspaces of $\bC^{2q}$. The lemma above implies readily that if $V$ is such a vector space, then the pointed symmetric space $(\G_q(\bC^{2q}), V)$ has only one pole, which is $V^\perp$, the orthogonal complement of $V$ in
$\bC^{2q}$ relative to the usual Hermitian inner product.  }
\end{remark}

As a next step, we look at s-centrioles in $\G_q(\bC^{2q})$. Since $\G_q(\bC^{2q})$ is an irreducible and
simply connected symmetric space, there is a {\it unique}  
s-centriole of $(\G_q(\bC^{2q}), A_q)$ relative to the pole $-A_q$
(see Theorem 1.2 and the subsequent remark in \cite{Ma-Qu}).
To describe it,
we first find  a shortest geodesic segment from $A_q$ to
 $-A_q$ in $\G_q(\bC^{2q})$. 
  Let us consider the curve $\gamma: [0, 1]\to \G_q(\bC^{2q})\subset \U_{2q}$,
$$
\gamma(t)=
\exp\left[t
\left(%
\begin{array}{ccccccc}
   0& \frac{\pi  i}{2} I_q \\
 \frac{\pi i}{2} I_{q} & 0
\end{array}%
\right)\right] .A_q
=
\left(%
\begin{array}{ccccccc}
   \cos (\frac{\pi t}{2})I_q& i\sin (\frac{\pi t}{2})I_q \\
  i\sin (\frac{\pi t}{2}) I_q & \cos (\frac{\pi t}{2})I_q
\end{array}%
\right) .A_q,
$$
where the dot indicates the conjugation action. Observe that
$\gamma(0)= A_q$ and $\gamma(1)=-A_q$.
We claim that $\gamma$ is a shortest geodesic segment between $A_q$ and $-A_q$
in $\G_q(\bC^{2q})$.
Indeed, for any $t\in [0,1]$ the matrix $\gamma'(t)$ is $\U_{2q}$-conjugate with
the Lie bracket of the matrices
$$
\left(%
\begin{array}{ccccccc}
   0& \frac{\pi  i}{2} I_q \\
 \frac{\pi i}{2} I_{q} & 0
\end{array}%
\right)
$$
and $A_q$, which is equal to
$$
\left(%
\begin{array}{ccccccc}
   0&{\pi  i} I_q \\
 {\pi i} I_{q} & 0
\end{array}%
\right).
$$
Thus the length of $\gamma$ relative to the bi-invariant metric on $\U_{2q}$ given by Equation
(\ref{lega}) is equal to $\pi\sqrt{2q}$, which means that $\gamma$ is a shortest
path in $\U_{2q}$ between $A_q$ and $-A_q$ (see \cite[p.~127]{Mi} or Lemma \ref{length} below). Since the length of the vector $\gamma'(t)$ is independent of $t$,
$\gamma$ is a geodesic segment. Its midpoint is
\begin{equation}\label{unupedoi}\gamma\left(\frac{1}{2}\right)
=
\left(%
\begin{array}{ccccccc}
   0&  I_q \\
 - I_{q} & 0
\end{array}%
\right).
\end{equation}
In view of Lemma \ref{pol}, the centriole we are interested in is the orbit of $\gamma(\frac{1}{2})$
under the $K_e$-action. Since $K_e=(\U_q\times\U_q)/Z(\U_{2q})$, this is the same as the orbit of $\gamma(\frac{1}{2})$ under conjugation by $\U_q\times\U_q\subset
\U_{2q}$. One can easily see that this orbit consists of all matrices of the form
$$\left(%
\begin{array}{ccccccc}
   0&  -C^{-1} \\
 C & 0
\end{array}%
\right)$$
where  $C$ is in $\U_q$. 
Multiplication from the left by the matrix given by Equation (\ref{unupedoi})
induces an isometry between the latter orbit and the subspace
of $\U_{2q}$ formed by all matrices
 $$\left(%
\begin{array}{ccccccc}
   C&  0 \\
 0 & C^{-1}
\end{array}%
\right),$$
with $C\in \U_q$.

We deduce that if we equip the s-centriole of $(\G_q(\bC^{2q}), A_q)$
relative to $-A_q$ with the submanifold metric,  then it becomes isometric to $\U_q$, where the latter is endowed with the
bi-invariant metric induced by
\begin{equation}\label{xyz}\langle X,Y\rangle =-2{\rm  tr(}XY),
\end{equation}
$X,Y \in \u_q$.
Moreover, if instead of $A_q$ the base point is an arbitrary element $A$ of $\G_q(\bC^{2q})$, then
the only pole of $(\G_q(\bC^{2q}),A)$ is the matrix $-A$.
The corresponding centriole is obtained from the previous one by
conjugation with $B$, where $B\in \U_{2q}$ satisfies
$A=BA_qB^{-1}$. Thus this centriole has the same isometry type as the
previous one.

\begin{remark}\label{minus}{\rm  We saw that there is a natural isometric identification between the centriole
of $(\G_q(\bC^{2q}), A)$ relative to $-A$ and $\U_q$. One can also see from the previous considerations that this centriole is invariant under the isometry
$X\mapsto -X$, $X\in \U_{2q}$, and the isometry induced on $\U_q$ is
$X'\mapsto -X'$, $X'\in \U_q$. }
\end{remark}

\subsection{The $\U$-Bott chain}\label{ubott} The following chain of inclusions results from the previous two subsections.
We start with $\tilde{P}_0:=\U_{2q}$, equipped with the bi-invariant Riemannian metric
defined by Equation (\ref{lega}).
The top-dimensional s-centriole of $(\tilde{P}_0, I)$ relative to $-I$ is denoted by $\tilde{P}_1$.
Pick $J_1\in \tilde{P}_1$. (The reason why the elements of $\tilde{P}_1$ are denoted by $J$
is explained in Appendix \ref{apa}, particularly Definition \ref{wat} and Equation
(\ref{was}).) By Remark \ref{remra},  $-J_1$ is in $\tilde{P}_1$, too.
We denote by $\tilde{P}_2$ the s-centriole of $(\tilde{P}_1, J_1)$ relative to the pole
$-J_1$. We have
$$\tilde{P}_0 \supset \tilde{P}_1 \supset \tilde{P}_2.$$
The elements of the chain are described by the following isometries:
$$\tilde{P}_0\simeq \U_{2q}, \ \tilde{P}_1\simeq \G_q(\bC^{2q}), \
\tilde{P}_2 \simeq \U_q,$$
where $\G_q(\bC^{2q})$ carries the (symmetric space) metric induced via its embedding in
$\U_{2q}$ and    $\U_q$ is endowed with the  metric described by Equation (\ref{xyz}).

We now take $q=8n$ and repeat the construction  above three more times.
By always choosing the top-dimensional centriole, we ensure that all our spaces
are invariant under the map $\U_{16n}\to \U_{16n}$, $X\mapsto -X$
(see Remarks \ref{remra} and \ref{minus} above).
We proceed as follows:

First we pick $J_2 \in \tilde{P}_2$ as a base point.
Then $-J_2$ is a pole of $(\tilde{P}_2,J_2)$. Indeed, we know that
 the geodesic symmetries $s_{J_2}$ and
$s_{-J_2}$ of $\U_{16n}$ are equal (see Example \ref{anyc}) and $\tilde{P}_2$ is a totally geodesic
submanifold of $\U_{16n}$.

After that, we consider the top-dimensional s-centriole
of $(\tilde{P}_2, J_2)$ relative to $-J_2$ and denote it by $\tilde{P}_3$.
As before, we have the identification
 $$\tilde{P}_3\simeq \G_{4n}(\bC^{8n}).$$

 In the same way, we construct  $\tilde{P}_4, \ldots, \tilde{P}_8$,
 by picking $J_{k-1}$ in $\tilde{P}_{k-1}$ and defining
 $\tilde{P}_k$ as the top-dimensional centriole of $(\tilde{P}_{k-1},J_{k-1})$ relative to
 $-J_{k-1}$, for all $k=4,\ldots, 8$. We have the identifications:
 $$\tilde{P}_5\simeq \G_{2n}(\bC^{4n}), \
\tilde{P}_6\simeq \U_{2n}, \
\tilde{P}_7\simeq \G_n(\bC^{2n}), \
\tilde{P}_8\simeq \U_n,$$
where each $\tilde{P}_k$ carries the submanifold metric.
Similarly to Equation (\ref{xyz}), one can see that the Riemannian
metric induced on $\U_n$ via the diffeomorphism $\tilde{P}_8\simeq \U_n$
coincides with the bi-invariant metric on $\U_n$ induced by
\begin{equation}\label{xir}\langle X,Y\rangle =-16{\rm tr}(XY),\end{equation}
$X,Y\in \u_n$.

In this way we have constructed the $\U$-Bott chain, which  is
$\tilde{P}_0\supset \tilde{P}_1 \supset \ldots \supset \tilde{P}_8$.

\subsection{Bott's periodicity theorems}\label{proofbo}
Bott's original proof (see \cite{Bo})  uses the space of paths between two points in a Riemannian manifold.
\begin{definition}
If $M$ is a Riemannian manifold and $p,q$ are two points in $M$, we denote by $\Omega(M;p,q)$ the space of
piecewise smooth paths $\gamma : [0,1]\to M$ with $\gamma(0)=p$ and
$\gamma(1)=q$.
\end{definition}
The space $\Omega(M;p,q)$ has a topology which is induced by a certain canonical metric
(the details can be found for instance in \cite[Section 17]{Mi}).

Let $(P,o)$ be again a pointed compact symmetric space, $p$
a pole of it, and $Q\subset P$ one of the corresponding s-centrioles.
Recall that $Q$ consists of midpoints of geodesics in $P$ from
$o$ to $p$. We have a continuous injection
\begin{equation}\label{starj} j: Q\to \Omega(P;o,p)\end{equation}
that assigns to $q\in Q$ the unique shortest geodesic segment $[0,1]\to M$ from $o$ to $p$
whose midpoint is $q$.
 This
 induces a map
$$j_*:\pi_i(Q)\to \pi_i(\Omega(P; o,p)) =\pi_{i+1}(P)$$
between homotopy groups.   Bott's proof \cite{Bo} relies on the fact that this map is an isomorphism for all $i> 0$ that are smaller than a certain number which can be calculated explicitly in concrete situations, including all the situations we have
described in Subsections \ref{sochain}, \ref{spch}, and \ref{ubott}. The main tool is Morse theory, see also Milnor's book \cite{Mi} (for a different approach we address to
\cite{Mit}).

We now apply the result above for the elements of the $\S\O$-chain,
see Subsection \ref{sochain}. For all $i=1,2,\ldots$ sufficiently smaller than $n$, we obtain
$$\pi_i(\SO_{n})=\pi_i(\tP_8)
\simeq \pi_{i+1}(\tP_7) \simeq \ldots \simeq \pi_{i+7}(\tP_1)
\simeq \pi_{i+8}(\tP_0) =\pi_{i+8}(\SO_{16n}).$$
This yields the following isomorphism between stable homotopy groups:
$$\pi_k(\O)\simeq \pi_{k+8}(\O),$$
for all $k=0,1,2,\ldots$.
This is Bott's periodicity theorem for the orthogonal group.
Similarly, for the  unitary and symplectic groups, one has
$$\pi_k(\U)\simeq \pi_{k+2}(\U) \quad  {\rm and} \quad  \pi_k(\Sp)\simeq \pi_{k+8}(\Sp)$$
for all $k=0,1,2, \ldots$.

\section{Inclusions between Bott chains}\label{inclubot}

In this section we link the three Bott chains
 constructed above.
The following  lemmata are  key ingredients that make this process possible.
We recall that for any $q\ge 1$ the Lie group $\U_{2q}$ carries the bi-invariant Riemann metric described by Equation (\ref{lega}).
We regard $\S\O_{2q}$ as a Lie subgroup of $\U_{2q}$ and endow it with the
submanifold metric (note that for $q=8n$ this is the same as the metric
described by Equation (\ref{metrcsos})).
For  $r\ge 1$ we also consider the symplectic group $\Sp_r$, which is defined as the space of all $\bH$-linear automorphisms of $\bH^n$ that preserve the norm of a vector.
As explained in Subsection \ref{5e},  this group has a canonical embedding  into
$\U_{2r}$. More precisely, $\Sp_r$  can be  identified  with the subgroup of $\U_{2r}$ that consists of all matrices of the form
$$
\left(
\begin{array}{cccccc}
A & -\overline{B}  \\
B & \overline{A}\end{array}
\right)
$$
which are in $\U_{2r}$, where $A$ and $B$ are $r\times r$ matrices with complex entries
(see \cite[Ch.~I, Section 1.11]{Br-tD}).
Yet another canonical embedding, which we also need here, is the one of  $\U_r$  into $\Sp_r$, see Subsection \ref{1em}. Concretely, $\U_r$
can be considered as the subgroup of $\Sp_r$ consisting of all matrices which are of the above form with $B=0$ and $A\in \U_r$.

For future use we also mention that  $\Sp_r$ lies in $\U_{2r}$
and  $\U_{r}$ lies in $\Sp_r$   as fixed point sets of certain involutive group automorphisms.
More precisely, let us consider the element
$$
K_r:=
\left(
\begin{array}{cccccc}
 0 & I_r  \\
-I_r & 0 \end{array}
\right)
$$
of $\U_{2r}$ and the  group automorphism of 
$\U_{2r}$ given by $X\mapsto K_r \overline{X}K_r^{-1}$,
where $\overline{X}$ is the complex conjugate of $X$:
the automorphism is involutive and its fixed point set is just $\Sp_r$. 
In the same vein,  let us consider the element
$$
A_r:=
\left(
\begin{array}{cccccc}
 iI_{r} & 0  \\
0 & -iI_{r} \end{array}
\right)
$$
of $\Sp_{r}$ and the corresponding (inner) automorphism
  of $ \Sp_{r}$,
$\bt(X):= A_r X A_r^{-1}$: 
this automorphism is involutive as well and its fixed point set  is equal to $\U_{r}$
(note that $A_r$ has also been used in Subsections \ref{poled} and \ref{poleun} and is also relevant in Subsection  \ref{7em}).

Let us now consider the inner product on $\u_{2r}$ given by
$$\langle X, Y\rangle = -\frac{1}{2}{\rm tr(}XY),\quad X, Y\in \u_{2r}.$$ 
Note that the bi-invariant Riemannian metric induced on
$\U_{2r}$ is different from the one defined by
Equation (\ref{lega}). However, we are exclusively interested in the subspace metrics 
 on  $\Sp_r$ and $\U_r$. On the last space,  the induced metric  is bi-invariant and satisfies
$$\langle X,Y\rangle =-{\rm tr(}XY),$$ for all $X, Y\in \u_{r},$
i.e.~this metric is the one given by Equation (\ref{lega}).

\begin{lemma}\label{length} Relative to the metrics defined above, we have:
\begin{align*}
{}& {\rm dist}_{\SO_{2q}}(I, -I)={\rm dist}_{\U_{2q}}(I,-I)=\pi\sqrt{2q}, \\
{}& {\rm dist}_{\U_r}(I, -I)={\rm dist}_{\Sp_r}(I,-I)=\pi\sqrt{r}.
\end{align*}
\end{lemma}

\begin{proof} The length of a shortest geodesic segment in $\U_{2q}$ between $I$ and $-I$ has been calculated in \cite[Section 23]{Mi}.
It is equal to $\pi\sqrt{2q}$. By \cite[Section 24]{Mi}, a shortest geodesic segment in
$\SO_{2q}$ from $I$ to $-I$ is
$$[0,1]\to \SO_{2q}, \
t\mapsto \exp\left[ t\pi
\left(
\begin{array}{cccccc}
0 &1  & \ldots  & 0 & 0  \\
-1 & 0 &\ldots  & 0 & 0 \\
& & \ddots & & \\
 0&0&\ldots & 0 & 1\\
 0 & 0 & \ldots & -1 & 0
\end{array}
\right)
\right].
$$
Its length is also equal to  $\pi\sqrt{2q}$.

To justify the second equation in the lemma, we just note that 
$$[0,1]\to \U_{2r}, \ t\mapsto
\exp\left[ t
\left(
\begin{array}{cccccc}
\pi iI_r & 0  \\
0 & -\pi iI_r
\end{array}
\right)
\right]
$$
is 
a shortest geodesic segment in
$\U_{2r}$ from $I$ to $-I$.
The image of this geodesic lies entirely in $\U_r \subset \Sp_r$ and is consequently shortest in both $\U_r$ and $\Sp_{r}$. Its length can be calculated as before, by using
\cite[Section 23]{Mi}.
\end{proof}

The next lemma concerns the $\S\O$-Bott chain, which has been constructed
in Subsection \ref{sochain}. The result can be found in \cite[p.~137]{Mi}.
Since it plays an important role in our development, we state it separately.

\begin{lemma}\label{equipe} If we equip each $P_k$, $k=1,2,\ldots, 7$ with the submanifold metric,
then we have
$${\rm dist}_{\SO_{16n}}(I,-I)=
{\rm dist}_{P_1}(J_1,-J_1)=\ldots=
{\rm dist}_{P_7}(J_7,-J_7).$$
\end{lemma}
This result can also be deduced from \cite{Qu-Su}.
Relevant to this context is 
  \cite[Remark 3.2 b)]{Na-Su}, too.

We are now ready to construct the inclusions between the three Bott chains.

\subsection{Including $P_k$ into  $\tilde{P}_k$}\label{proo}
We start by recalling   that  $P_1$ is one of the two s-centrioles
of $(\SO_{16n},I)$ relative to the pole $-I$ (see  Subsection \ref{sochain}).
  Also recall  that $\tilde{P}_1$ is the top-dimensional s-centriole
of $(\U_{16n},I)$ relative to the pole $-I$ (see Subsection \ref{ubott}).
By Lemma \ref{length}, ${P}_1$ is contained in one of the s-centrioles
 of $(\U_{16n},I)$ relative to $-I$, call it $\tilde{P}'_1$.

 \noindent{\it Claim.} $\tilde{P}'_1=\tilde{P}_1$, i.e.~$P_1\subset \tilde{P}_1$.

 \noindent 
 Both $J_1$ and $-J_1$ are in $\tP_1$, thus also in $\tilde{P}_1'$.
 The geodesic symmetries $s_{J_1}$ and $s_{-J_1}$ of $\U_{16n}$ are equal.
 Since $\tilde{P}_1'$ is a totally geodesic submanifold of $\U_{16n}$,
 the restrictions of the two geodesic symmetries to $\tilde{P}'_1$ are equal, too.
 Therefore, $-J_1$ is a pole of the pointed symmetric space $(\tilde{P}'_1,J_1)$. On the other hand, $\tilde{P}'_1$ is isometric to one of the symmetric spaces
 $\G_k(\bC^{16n})$, where
 $0\le k \le 16n$ (see Subsection \ref{poled}).
 It is known that amongst these Grassmannians there is just one which admits a pole
 relative to a given base point,
 namely the one corresponding to $k=8n$ (see Remark \ref{notan}).  This
 finishes the proof of the claim.

Note that the following diagram is commutative:
\begin{equation*}\label{ppseq}
\vcenter{\xymatrix{
\tP_1\ar[d]^{\cap} \ar[r]^{{\jmath}_1  \ \ \ \ \ \  }&
\Omega({P}_0; I, -I)\ar[d]^{\cap} &
\\
\tilde{P}_1\ar[r]^{\tilde{\jmath}_1  \ \ \ \ \ \ } &
\Omega(\tilde{P}_0; I, -I) &\\
}}
\end{equation*}
where the horizontal arrows are inclusion maps and the vertical arrows are given by
Equation (\ref{starj}).

Recall that $P_2$ is an s-centriole of
 $(\tP_1,J_1)$ relative to $-J_1$.
 By Lemmata \ref{length} and \ref{equipe},
 any shortest geodesic segment in $\tP_1$ which joins $J_1$ and
 $-J_1$  is also shortest  in $\tilde{P}_1$.
 Since $\tilde{P}_2$ is the unique   s-centriole of $(\tilde{P}_1, J_1)$ relative to
 $-J_1$   (see Subsection \ref{poleun}), we have
\begin{equation}\label{p2i}\tP_2\subset \tilde{P}_2.\end{equation}
Again, we have a commutative diagram, which is:
 \begin{equation*}
\vcenter{\xymatrix{
\tP_2\ar[d]^{\cap} \ar[r]^{\jmath_2  \ \ \ \ \ \  }&
\Omega(\tP_1; J_1, -J_1)\ar[d]^{\cap} &
\\
\tilde{P}_2\ar[r]^{\tilde{\jmath}_2  \ \ \ \ \ \ } &
\Omega(\tilde{P}_1; J_1, -J_1) &\\
}}
\end{equation*}

 In the same way we prove that we have the inclusions
\begin{equation}\label{pki}\tP_k\subset \tilde{P}_k\end{equation}
 for all $k=3, \ldots, 8$.

 \subsection{The inclusion $\tP_k \subset \tilde{P}_k$ as fixed points of the complex conjugation}\label{fixedpoints}
We will use the following notations.

\noindent{\bf Notations.} Let $A$ be a topological space.
If $a$ is an element of $A$,
then $A_a$ denotes the connected component of $A$ which contains $a$.
If $\sigma$ is a map from  $A$ to $A$ then $A^\sigma:=\{x\in A  \ : \ \sigma(x)=x\}$.

The main tool we will use in this subsection is the following lemma.

\begin{lemma}\label{lettp}
Let $(\tilde{P}, o)$ be a compact connected pointed symmetric space,
$p$ a pole of $(\tilde{P},o)$ and $\gamma_0 : [0,1]\to \tilde{P}$ a geodesic segment which is shortest between  ${\gamma}_0(0)=o$ and ${\gamma}_0(1)=p$.
Set $j_0:=\gamma_0\left(\frac{1}{2}\right)$ and denote by ${\tilde{Q}}$ the centriole of
$(\tilde{P}, o)$ relative to $p$ which contains $j_0$ (see Definition \ref{pppole}). Let also  $\sigma$  be
an isometry of $\tilde{P}$. Assume that $\sigma(o)=o$, $\sigma(p)=p$,  and
set  $\tP:=(\tilde{P}^\sigma)_o$. Also assume that the trace of $\gamma_0$
is contained in $\tP$.
Then:

(a) $p$ is a pole of $(P,o)$,

(b) $\tilde{Q}$ is $\sigma$-invariant,

(c) $(\tilde{Q}^\sigma)_{j_0}=(C_p(P,o))_{j_0}$.
\end{lemma}

\begin{proof} (a) Since $p$ is a pole of $(\tilde{P}, o)$, the geodesic
reflections $s^{\tilde{P}}_o$ and $s^{\tilde{P}}_p$ are equal.
But $P$ is a totally geodesic submanifold of $\tilde{P}$, hence
the geodesic
reflections $s^P_o=s^{\tilde{P}}_o|_P$ and $s^P_p=s^{\tilde{P}}_p|_P$ are equal
as well.

(b) Take $x\in C_p(\tilde{P},o)$. Then there exists a geodesic segment
$\gamma:[0,1]\to \tilde{P}$ with $\gamma(0)=o$, $\gamma(1)=p$, and
$\gamma\left(\frac{1}{2}\right)=x$. The path $\sigma\circ \gamma :[0,1]\to \tilde{P}$
is also a geodesic segment. It joins $\sigma\circ\gamma(0)=o$ with
$\sigma\circ\gamma(1)=p$. Thus, its midpoint $\sigma(x)$  lies in
$C_p(\tilde{P},o)$ as well. We have shown that $\sigma$ leaves
$C_p(\tilde{P},o)$ invariant and induces  a homeomorphism of
it. Consequently, $\sigma$ maps $P= C_p(\tilde{P},o)_{j_0}$
onto a connected component of $C_p(\tilde{P},o)$. This must be $C_p(\tilde{P},o)_{j_0}$,
because $\sigma(j_0)=j_0$.

(c) Since $\tP$ is a totally geodesic submanifold of $\tilde{P}$,
we deduce that $C_p(P,o)\subset C_p(\tilde{P},o)\cap \tilde{P}^\sigma$,
hence $C_p(P,o)_{j_0}\subset C_p(\tilde{P},o)_{j_0}\cap \tilde{P}^\sigma=
\tilde{Q}^\sigma$. We have shown that
$C_p(P,o)_{j_0}\subset (\tQ^\sigma)_{j_0}.$

Let us now prove the opposite inclusion. Take $j$ an arbitrary element of
${\tQ}^\sigma$.
There exists ${\gamma} : [0,1]\to \tilde{P}$ a geodesic segment with
${\gamma}(0)=o$, ${\gamma}(1)=p$, and ${\gamma}\left(\frac{1}{2}\right)=j$.
Since $\tilde{Q}$ is an s-centriole, we can assume that ${\gamma}$ is shortest between $o$ and $p$.
This implies that the restriction of ${\gamma}$ to the interval $\left[0,\frac{1}{2}\right]$
is a shortest geodesic segment between $o$ and $j$;
moreover, it is the {\it unique} shortest geodesic  segment $\left[0,\frac{1}{2}\right]
\to \tilde{P}$ between $o$ and $j$ (cf. e.g. \cite[Corollary 2.111]{Ga-Hu-La}).
On the other hand, the curve $\sigma \circ \gamma : \left[0,\frac{1}{2}\right]
\to \tilde{P}$ is a shortest geodesic segment with the properties
$$\sigma\circ{\gamma}(0)=\sigma(o)=o \quad {\rm and} \quad
\sigma \circ {\gamma} \left(\frac{1}{2}\right)
= \sigma(j)=j.$$
Consequently, we have $\sigma \circ {\gamma} ={\gamma}$, and therefore
the trace of ${\gamma}$ is contained in $P$. This implies that $j\in C_p({P},o)$.
We have shown that  ${\tQ}^\sigma\subset C_p(P,o)$.
This implies readily the desired conclusion. \end{proof}

Let us denote by $\tau$ the (isometric) group automorphism of $\U_{16n}$ given by complex conjugation.
That is,
$\tau: \U_{16n} \to \U_{16n}$, $$ \tau(X):=\overline{X}, \quad X \in \U_{16n}.$$
Note that the fixed point set of $\tau$ is $\O_{16n}$.

By collecting results we have proved in this subsection and the previous one, we can now state
the following theorem.

\begin{theorem}\label{fmain} 
For any $k\in \{0,1,\ldots, 8\}$, the space $\tilde{P}_k$ is $\tau$-invariant and we have  \begin{equation}\label{tPk}\tP_k=(\tilde{P}_k^\tau)_{J_k}.\end{equation}
The 
following diagram is commutative:
 \begin{equation}\label{exactse}
\vcenter{\xymatrix{
\tP_0\ar[d]^{\cap} &
\tP_1\ar[d]^{\cap}
\ar[l]_{  \supset  } &
\tP_2
\ar[l]_{\supset} \ar[d]^{\cap}  & \ar[l]_{\supset}\cdots&
\tP_8\ar[l]_{\supset}\ar[d]^{\cap} &
\\
\tilde{P}_0 &
\tilde{P}_1
\ar[l]_{  \supset  } &
\tilde{P}_2
\ar[l]_{\supset}  & \ar[l]_{\supset}\cdots&
\tilde{P}_8\ar[l]_{\supset} &\\
}}
\end{equation}
where the two horizontal components are the $\SO$- and the $\U$-Bott chains,
and the vertical arrows  are the  inclusions $P_k\subset \tilde{P}_k$,
$k\in \{0,1,\ldots, 8\}$, induced by Equation (\ref{tPk}).
The following diagram is also commutative
 \begin{equation}\label{exact}
\vcenter{\xymatrix{
\tP_{\ell+1}\ar[d]^{\cap} \ar[r]^{\jmath_{\ell+1}  \ \ \ \ \ \  }&
\Omega(\tP_\ell; J_\ell, -J_\ell)\ar[d]^{\cap} &
\\
\tilde{P}_{\ell+1}\ar[r]^{\tilde{\jmath}_{\ell+1}  \ \ \ \ \ \ } &
\Omega(\tilde{P}_\ell; J_\ell, -J_\ell) &\\
}}
\end{equation}
where  the maps $\jmath_{\ell+1}$ and $\tilde{\jmath}_{\ell+1}$ are the
canonical inclusions given by Equation (\ref{starj}),  for all $\ell\in \{0,1,\ldots, 7\}$.
\end{theorem}

\subsection{The inclusions $\tilde{P}_k\subset \bar{P}_k$}
We start with the standard inclusion
$\tilde{P}_0=\U_{16n}\subset \Sp_{16n}=\bar{P}_0$.
The $\Sp$-Bott chain defined in Subsection \ref{sochain} can be described in terms of
the complex structures $J_1, \ldots, J_8\in \SO_{16n}$ above as follows:
$\bar{P}_{k+1}$ is an s-centriole of  $(\bar{P}_k, J_k)$ relative to $-J_k$, for all $k=0,1,\ldots, 7$ (as already mentioned in Subsection \ref{sochain}, the main reference for this construction is \cite[Section 24]{Mi};  see also \cite{Es}, Section 19, especially pp.~43--44).
With the methods of Subsection \ref{proo} one can show that
we have the totally geodesic embeddings
$\tilde{P}_k\subset \bar{P}_k$, for all $k=0,1,\ldots, 8$.

As mentioned at the beginning of this section,  $\U_{16n}$ lies in $\Sp_{16n}$ as
the fixed point set of the (involutive, inner) group automorphism 
  $\bt : \Sp_{16n}\to \Sp_{16n}$,
$\bt(X):= A_{8n} X A_{8n}^{-1}$.
In the same way as in Subsection \ref{fixedpoints}, we can prove the following
analogue of Theorem \ref{fmain}:

\begin{theorem}\label{mmain} For any $k\in \{0,1,\ldots, 8\}$, the space $\bar{P}_k$ is $\bar{\tau}$-invariant and we have  
\begin{equation}\label{bPk}\tilde{P}_k=(\bar{P}_k^{\bar{\tau}})_{J_k}.\end{equation}
The 
following diagram is commutative:
 \begin{equation}\label{exacts}
\vcenter{\xymatrix{
\tilde{P}_0\ar[d]^{\cap} &
\tilde{P}_1\ar[d]^{\cap}
\ar[l]_{  \supset  } &
\tilde{P}_2
\ar[l]_{\supset} \ar[d]^{\cap}  & \ar[l]_{\supset}\cdots&
\tilde{P}_8\ar[l]_{\supset}\ar[d]^{\cap} &
\\
\bar{P}_0 &
\bar{P}_1
\ar[l]_{  \supset  } &
\bar{P}_2
\ar[l]_{\supset}  & \ar[l]_{\supset}\cdots&
\bar{P}_8\ar[l]_{\supset} &\\
}}
\end{equation}
 where the two horizontal components are the $\U$- and the $\Sp$-Bott chains,
and the vertical arrows  are the inclusions $\tilde{P}_k\subset \bar{P}_k$,
$k\in \{0,1,\ldots, 8\}$, induced by Equation (\ref{bPk}).
The following diagram is also commutative
\begin{equation}\label{exact}
\vcenter{\xymatrix{
\tilde{P}_{\ell+1}\ar[d]^{\cap} \ar[r]^{\tilde{\jmath}_{\ell+1}  \ \ \ \ \ \  }&
\Omega(\tilde{P}_\ell; J_\ell, -J_\ell)\ar[d]^{\cap} &
\\
\bar{P}_{\ell+1}\ar[r]^{\bar{\jmath}_{\ell+1}  \ \ \ \ \ \ } &
\Omega(\bar{P}_\ell; J_\ell, -J_\ell) &\\
}}
\end{equation}
where the maps $\tilde{\jmath}_{\ell+1}$ and $\bar{\jmath}_{\ell+1}$ are the
canonical inclusions given by Equation (\ref{starj}),  for all $\ell\in \{0,1,\ldots, 7\}$.
\end{theorem}

\begin{remark} {\rm We note in passage that all maps in the commutative diagrams described by Equations 
(\ref{exactse}) and (\ref{exacts}) are inclusions of reflective submanifolds.}
\end{remark}

\section{Periodicity of inclusions between Bott chains}\label{bottch}

\subsection{The inclusion $\tP_8\subset \tilde{P}_8$}\label{inclus}

We have the isometries:
$$\tP_8\simeq \SO_{n} \ {\rm and \ } \tilde{P}_8 \simeq \U_{n}.$$
The first is discussed in Proposition \ref{isotype} (b) and the second
in Subsection \ref{ubott}.
Note that $\tilde{P}_1$ is actually contained in $\SU_{16n}$ (see  Subsection
\ref{poleun}).
Thus, from Theorem \ref{fmain} we obtain  the following commutative diagram:  \begin{equation*}
\vcenter{\xymatrix{
\SO_{16n}\ar[d]^{\cap} &
\tP_8\ar[l]_{\supset}\ar[d]^{\cap} &
\\
\SU_{16n} &
\tilde{P}_8\ar[l]_{\supset} &\\
}}
\end{equation*}
where all arrows are inclusion maps, as follows: $P_8\subset P_0=\SO_{16n}$;
$\tilde{P}_8\subset \tilde{P}_1\subset \SU_{16n}$;   $\SO_{16n}$ 
is contained in $\SU_{16n}$ as the identity component of the fixed point set
of $\tau$, the latter being the complex conjugation; finally, by Theorem \ref{fmain},
the space
$\tilde{P}_8$ is $\tau$-invariant and $\tP_8$ is a connected component of
the fixed point set $\tilde{P}_8^\tau$.
We will prove the following result.

\begin{theorem}\label{co}
There exists an isometry $\psi: \tilde{P}_8\to \U_n$
which maps ${P}_8$ to $\SO_n$ and makes the following diagram
commutative:
\begin{equation*}
\vcenter{\xymatrix{
{P}_8\ar[d]^{\cap} \ar[r]^{\psi|_{P_8} \ \ \ }&
 \SO_n\ar[d]^{\cap}
\\
\tilde{P}_8\ar[r]^{\psi} &
\U_n \\
}}
\end{equation*}
Here the inclusions $P_8\subset \tilde{P}_8$ and  $\SO_n\subset \U_n$ are the one mentioned in the diagram (\ref{exactse}), respectively    the standard one (see 
e.g.~Subsection
\ref{1e}).
\end{theorem}

The rest of this subsection is devoted to the proof of this theorem.
First pick $J\in {P}_8$ and denote $${\p}=T_J\tilde{P}_8.$$
Let
$R: \p \times \p \times \p \to \p$
be the curvature tensor of $\tilde{P}_8$ at the point $J$. It is a Lie triple in the sense of
Loos \cite[Vol.~I]{Lo}.  Let  $\c$ be the center of this Lie triple, that is,
$$\c=\{\eta \in \p \ : \ R(\eta, x)y= 0 \ {\rm for \ all \ } x,y\in \p\}.$$
We also denote by $\check{\p}$ the orthogonal complement of $\c$ in $\p$
relative to the Riemann metric $\langle \ , \ \rangle_J$  of $\tilde{P}_8$ at the point $J$.
Both elements of the splitting
$$\p=\c\oplus \check{\p}$$
are Lie subtriples of $\p$.
Recall from Subsection \ref{ubott} that there exists an isometry
$$\varphi: \tilde{P}_8\to \U_n,$$ where $\U_n$  is equipped with the 
bi-invariant Riemannian metric described by Equation  (\ref{xir}).
Thus, the center $\c$ is a 1-dimensional vector subspace of $\p$. Let $\tau_*: \p \to \p$ be the differential of $\tau|_{\tilde{P}_8}$ at $J$.
It is a Lie triple automorphism of $\p$ that preserves the inner product
$\langle  \ , \ \rangle_J$. Thus it leaves both the center $\c$
and its orthogonal complement $\check{\p}$ invariant.
The fixed point set of $\tau_*$, call it Fix$(\tau_*)$, is a Lie sub-triple which splits as:
$${\rm Fix}(\tau_*) = {\rm Fix}(\tau_*|_{\c}) \oplus {\rm Fix}(\tau_*|_{\check{\p}}).$$
The first term of the splitting above is contained in the center of
${\rm Fix}(\tau_*)$. On the other hand, $P_8$ is the connected
component of $J$ in
the fixed point set of $\tau|_{\tilde{P}_8} : \tilde{P}_8 \to \tilde{P}_8$.
Therefore we have
${\rm Fix}(\tau_*)=T_JP_8$; as $P_8$ is isometric to $\SO_n$
(see the beginning of this section),
$T_JP_8$   is isomorphic to the Lie triple of
$\SO_n$. The latter Lie triple has no center, since $\SO_n$ is a semi-simple
symmetric space. Consequently, we have
${\rm Fix}(\tau_*|_{\c})=\{0\}.$
Both $\tau$ and $\tau_*$ are involutive, thus
\begin{equation}\label{ta}\tau_*(x) = -x, \ {\rm for  \ all \ } x\in \c.\end{equation}
Consequently,
 $${\rm Fix}(\tau_*) = {\rm Fix}(\tau_*|_{\check{\p}}).$$
We denote by $\check{P}_8$ the complete connected totally geodesic subspace of
$\tilde{P}_8$ corresponding to the Lie sub-triple $\check{\p}$.
It  is mapped by $\varphi$ isometrically onto $\SU_n$, the latter being equipped with the restriction of the bi-invariant metric
given by  Equation (\ref{xir}). The space $\check{P}_8$ is $\tau$-invariant and we have
\begin{equation}\label{chc}(\check{P}_8^\tau)_J =(\tilde{P}_8^\tau)_J=\tP_8.
\end{equation}
We need the following lemma.

\begin{lemma}\label{ifwei}
 There exists an isometry $\varphi : \tilde{P}_8\to \U_{n}$ such that
$\varphi (J)=I_n$ and $\varphi(P_8)=\SO_n$.
Moreover, there exists $A\in \SU_n$ which satisfies $A=A^T$ such that
\begin{equation}\label{bara}\varphi(\tau(p))=A\overline{\varphi(p)}A^{-1},\end{equation} for all $p\in \check{P}_8$.
\end{lemma}

\begin{proof}
Let $\varphi: \tilde{P}_8\to \U_n$ be the isometry above. The condition $\varphi(J)=I_n$ is achieved after modifying  $\varphi$ suitably,
that is,  multiplying it pointwise by $\varphi(J)^{-1}$. This proves the first claim in the lemma.

We now prove the second claim. To this end, we first recall that 
$\varphi|_{\check{P}_8}: \check{P}_8\to \SU_n$ is an isometry,
where $\SU_n$ is equipped with the restriction of the bi-invariant metric
given by  Equation (\ref{xir}).
Thus, the map $\tau':= \varphi \circ \tau \circ \varphi^{-1}|_{\SU_n}$ is an
 involutive isometry of $\SU_n$. Moreover,
 the identity element $I_n$ is in the fixed point set $\SU_n^{\tau'}$.
 From Proposition \ref{gisom} we deduce that
there exists an involutive group automorphism $\mu$  of $\SU_n$ such that either
\begin{equation}\label{unu}\tau'(X)=\mu(X), \ {\rm for \ all \ } X \in \SU_n
\end{equation}
or
\begin{equation}\label{doi}\tau'(X)=\mu(X)^{-1},  \ {\rm for \ all \ } X \in \SU_n.
\end{equation}
Moreover, in the second case the space
 $(\SU_n^{\tau'})_{I_n}$ is isometric  to $\SU_n/\SU_n^\mu$, where the last space has the canonical symmetric space metric. Assume that we are in the second case.
 From Equation (\ref{chc}),   $\SO_n$
would be isometric  to   $\SU_n/\SU_n^\mu$. The  involutive group  automorphisms of
$\SU_n$ are classified, see e.g.~\cite[p.~281 and p.~290]{Wo}.
It turns out that  the group $\SU_n^\mu$ is isomorphic  to
$\S(\U_k\times \U_{n-k})$, for some $0\le k \le n$, or to
$\SO_n$, or to ${\rm Sp}_{n/2}$, if $n$ is divisible by 2.
None of the corresponding quotients is a symmetric space isometric to $\SO_n$.

We deduce that Equation (\ref{unu}) holds. Once again from the classification
of the involutive group automorphisms of $\SU_n$ mentioned above
(\cite[p.~290]{Wo}), we deduce readily the
presentation of $\tau$ described by Equation (\ref{bara}).
\end{proof}

 We are now ready to prove the main result of this subsection.

 \noindent {\it Proof of Theorem \ref{co}.} Let $\varphi : \tilde{P}_8 \to \U_n$ be
 the isometry mentioned in Lemma \ref{ifwei}.

 \noindent{\it Claim.} Equation (\ref{bara}) holds actually for all $p\in \tilde{P}_8$.

 Indeed, both $\varphi \circ \tau$ and $A\overline{\varphi}A^{-1}$ are isometries $\tilde{P}_8 \to \U_n$,
 which map $J$ to $I_n$. It remains to show that  their differentials at $J$ are  identically
 equal.  By Equation (\ref{bara}) they are  equal on the last component of the splitting
 $T_J\tilde{P}_8 = \c \oplus \check{\p}$. In fact, they are also equal on $\c$,
 in the sense that for any $x\in \c$ we have
  $$(d\varphi)_J\circ \tau_*(x)=A\overline{(d\varphi)_J(x)} A^{-1}.$$
  This can be justified as follows. First, by Equation (\ref{ta}), the left-hand side is equal to
  $-(d\varphi)_J(x)$. Second, since $\varphi : \tilde{P}_8 \to \U_n$ is an isometry,
  $(d\varphi)_J$ is a Lie triple isomorphism between $T_J\check{P}_8$ and
  $T_{I_n}\U_n$, thus it maps $x$ to the center of $T_{I_n}\U_n$, which is the space of
  all purely imaginary multiples of the identity; hence we have
  $\overline{(d\varphi)_J(x)}=-(d\varphi)_J(x)$ and this matrix commutes with $A$.

Let us now consider the map $c: \U_n \to \U_n$, $c(X)=A \overline{X} A^{-1}$, and observe that the following diagram is commutative:
\begin{equation*}
\vcenter{\xymatrix{
\tilde{P}_8\ar[d]^{\tau} \ar[r]^{\varphi}&
\U_n\ar[d]^{c} &
\\
\tilde{P}_8\ar[r]^{\varphi} &
\U_n &\\
}}
\end{equation*}
Since $\varphi(J)=I_n$, we deduce that $\varphi$ maps $(\tilde{P}_8^\tau)_J$ to $(\U_n)^c_{I_n}$.
The latter set, that is, the fixed point set of $c$, has been determined explicitly in 
\cite[p.~290]{Wo}: it is of the form $B \O_n B^{-1}$,
for some  $B \in \U_n$. The connected component of $I_n$ in this space is
$B\SO_n B^{-1}$. On the other hand, by Equation (\ref{chc}), we have $(\tilde{P}_8^\tau)_J=\tP_8$. Thus
$\varphi$ maps $\tP_8$ isometrically onto
$B \SO_n B^{-1}$.
In conclusion, the map $\psi: \tilde{P}_8 \to \U_n$,  $\psi(X) =B^{-1}\varphi(X)B$, has all the desired properties.
\hfill $\square$

\subsection{The inclusion $\tilde{P}_8\subset \bar{P}_8$}\label{sinclusion}
The following result is analogue to  Theorem \ref{co}:

\begin{theorem}\label{coc}
There exists an isometry $\chi: \bar{P}_8\to \Sp_n$
which maps $\tilde{P}_8$ to $\U_n$ and makes the following diagram
commutative:
\begin{equation*}
\vcenter{\xymatrix{
\tilde{P}_8\ar[d]^{\cap} \ar[r]^{\chi|_{\tilde{P}_8}\ \ \ }&
 \U_n\ar[d]^{\cap}
\\
\bar{P}_8\ar[r]^{\chi} &
\Sp_n \\
}}
\end{equation*}
Here the inclusions $\tilde{P}_8\subset \bar{P}_8$ and  $\U_n\subset \Sp_n$ are the one mentioned in the diagram (\ref{exacts}), respectively the standard one (see 
e.g.~Section
\ref{1em} and the beginning of Section \ref{inclubot}).
\end{theorem}

This can be proved by using the same method as in Subsection \ref{inclus}.
In fact, the proof is  even simpler in this case, since, unlike $\U_n$,
the symmetric space $\Sp_n$ is semisimple, i.e.~the corresponding Lie triple has
no center.

\begin{remark}\label{samesp} {\rm In the same spirit and with the same methods as  in Theorems \ref{co} and \ref{coc},
one can show that for the embeddings $P_4\subset \tilde{P}_4$ and
$\tilde{P}_4\subset \bar{P}_4$  one obtains commutative diagrams
\begin{equation*}
\vcenter{\xymatrix{
{P}_4\ar[d]^{\cap} \ar[r]^{\simeq}&
 \Sp_{2n}\ar[d]^{\cap}
\\
\tilde{P}_4\ar[r]^{\simeq} &
\U_{4n} \\
}} \quad \quad \quad \quad \quad \quad
\vcenter{\xymatrix{
\tilde{P}_4\ar[d]^{\cap} \ar[r]^{\simeq}&
 \U_{4n}\ar[d]^{\cap}
\\
\bar{P}_4\ar[r]^{\simeq} &
\SO_{8n} \\
}}
\end{equation*}
where the horizontal arrows indicate isometries.
More precisely, the spaces $P_4, \tilde{P}_4$, and $\bar{P}_4$ have the submanifold metrics
arising from the three Bott chains and the spaces
$\Sp_{4n}$, $\U_{8n}$, and $\SO_{8n}$ have the metrics described earlier
in this paper (see
the beginning of Section \ref{inclubot}) up to appropriate rescalings.
The inclusions $P_4\subset\tilde{P}_4$, $\tilde{P}_4 \subset \bar{P}_4$
are those
mentioned in the diagrams (\ref{exactse}) respectively (\ref{exacts}) and
the inclusions
$\Sp_{2n}\subset \U_{4n}$ and $\U_{4n}\subset \SO_{8n}$ are standard,
i.e.~those described in Subsections  \ref{5e}, respectively \ref{5em}.
  }
\end{remark}

\begin{remark}\label{ream} {\rm Assume that in the above context $n$ is divisible by $16$. As we have already pointed out
(see Remark \ref{sclaar} and Sections \ref{spch} and \ref{ubott}),   
each of the three Bott chains  can be extended using the centriole construction.
One obtains:
\begin{align*}
P_0\supset P_1\supset P_2 \supset \ldots \supset P_{16},\\
\tilde{P}_0\supset \tilde{P}_1\supset \tilde{P}_2 \supset \ldots \supset \tilde{P}_{16},\\
\bar{P}_0\supset \bar{P}_1\supset \bar{P}_2 \supset \ldots \supset \bar{P}_{16},
\end{align*}
where we have  isometries
$$P_{16}\simeq \SO_{n/16},\quad \tilde{P}_{16}\simeq \U_{n/16},
\quad \bar{P}_{16}\simeq \Sp_{n/16}.$$
 Theorems \ref{co} and \ref{coc} imply that the centriole constructions  can be 
performed in such a way that we have
$$P_k\subset \tilde{P}_k,\quad \tilde{P}_k\subset \bar{P}_k, \quad
8\le k \le 16,$$
and these inclusions are again those described by Tables 5 and 6,
up to some obvious changes of the subscripts.
This observation is one of the main achievements of our paper.  
We can express it in a more informal manner, by saying that 
 the inclusions $P_{k+8}\subset \tilde{P}_{k+8}, \tilde{P}_{k+8}
\subset \bar{P}_{k+8}$
are the same as $P_{k}\subset \tilde{P}_{k}$, respectively $\tilde{P}_{k}\subset 
\bar{P}_{k}$.}
\end{remark}

\section{Application: periodicity of maps between homotopy groups}\label{secper}

In this section we apply the main results of this paper, which are 
differential geometric, to the topology of classical Riemannian symmetric spaces. 
The results we prove here, i.e.~Theorems \ref{maintheo} and \ref{maitheo}, followed by Corollaries \ref{cor} and 
\ref{coragain},
  are in fact  just common knowledge in homotopy theory (one can prove them using techniques described e.g.~in \cite[Ch.~1]{May}).  
 The  goal of our approach is to provide more insight concerning these results, by 
indicating that there is  a differential geometric periodicity  phenomenon that stays 
  behind them, similar to the periodicity phenomenon that stays behind Bott's classical periodicity
  theorems \cite{Bo}.

We start by recalling that a simple application of the long exact homotopy sequence
of the principal bundle $\U_m\to \U_{m+1}\to \bS^{2m+1}$ shows
that the homotopy groups $\pi_i(\U_m)$ are $m$-stable.
More precisely, they remain  unchanged up to an isomorphism for any $m$ which is
larger than $\frac{i}{2}$. We denote by  $\pi_i(\U)$ the resulting group,
or rather, isomorphism class of groups.
The Bott periodicity theorem \cite{Bo} for the unitary group says that
$\pi_i( \U) = \pi_{i+2}(\U),$
for $i=0,1,2, \ldots$.
There is also a version of this result for the orthogonal and symplectic group.
First of all, we have
$\pi_i(\O_m)\simeq \pi_i(\O_{m+1})=:\pi_i(\O)$ for
all $m$ and $i$ such that $m\ge i+1$.
The periodicity theorem  in this case
says that
$\pi_i(\O)= \pi_{i+8}(\O)$, for  $i=0,1,2, \ldots$.
Similarly,
$\pi_i(\Sp)=\pi_{i+8}(\Sp)$, for $i=0,1,2, \ldots$ (see \cite[Section 24]{Mi}).

\subsection{The maps induced by  $\O_m\hookrightarrow \U_m$ }
Let us  consider  the canonical embedding map
$\imath_m: \O_m\hookrightarrow \U_m$. Let
$f^m_i:=(\imath_m)_*: \pi_i(\O_m)\to \pi_i(\U_m)$ be the map between homotopy groups
induced by $\imath_m$.
The following notion will be used in this section.

\begin{definition}\label{letapr} Let $\A, \A',\B$, and $\B'$ be groups and
$f: \A\to \B$, $f': \A'\to \B'$  group homomorphisms. We say that $f$ is {\it equivalent }
to $f'$ and denote $f\sim f'$ if there exist group isomorphisms  $g: \A\to \A'$ and $h: \B\to \B'$ that make the following diagram  commutative:
$$
\vcenter{\xymatrix{
\A
\ar[r]^{f}\ar[d]^{g} &
\B
 \ar[d]^{h}
\\
\A'
\ar[r]^{f'}  &
\B'
\\
}}
$$
\end{definition}
We will need the following result.

\begin{lemma}\label{equivc}
 The equivalence class modulo
$\sim$ of the map $f^m_i: \pi_i(\O_m)\to \pi_i(\U_m)$  is
stable. That is, modulo the equivalence relation $\sim$, the map $f^m_i$  is independent
of $m$ for all $m\ge i+1$.
\end{lemma}
\begin{proof} Let us consider the commutative diagram
$$
\vcenter{\xymatrix{
\O_m
\ar[r]^{\imath_m}\ar[d] &
\U_m
 \ar[d]
\\
\O_{m+1}
\ar[r]^{\imath_{m+1}} &
\U_{m+1}
\\
}}
$$
The vertical arrows indicate the canonical inclusion maps, given by
$$A\mapsto
\left(%
\begin{array}{ccccccc}
1 & 0 \\
  0 & A \end{array}%
\right)
$$
for any $m\times m$ orthogonal matrix $A$.
By functoriality we obtain the following commutative diagram.
$$
\vcenter{\xymatrix{
\pi_i(\O_m)
\ar[r]^{f^m_i}\ar[d] &
\pi_i(\U_m)
 \ar[d]
\\
\pi_i(\O_{m+1})
\ar[r]^{f^{m+1}_i} &
\pi_i(\U_{m+1})
\\
}}
$$
We only need to recall that for any $m \ge i+1$ both vertical arrows are isomorphisms
(to show that the map $\pi_i(\O_m)\to \pi_i(\O_{m+1})$ is an isomorphism for $m\ge i+1$, one uses the long exact sequence of the principal bundle
$\O_m \to \O_{m+1}\to \bS^m$).
\end{proof}

 Let us denote by $f_i$ the {\it equivalence class} of the map
$f^m_i$, for $m\ge i+1$.
Before stating the main result of this subsection, let us note that both the domain and the
codomain of the map $f^m_i:\pi_i(\O_m)\to \pi_i(\U_m)$ are periodic relative to $i$, with period equal to $8$.
The following theorem says
 that the map
$f^m_i$ itself is periodic (modulo $\sim$).

\begin{theorem}\label{maintheo} We have $f_i= f_{i+8}$, for all $i \ge 0$. \end{theorem}

\begin{proof}
Let us first assume that $i> 0$.
We use the notations which have been established in
the previous sections.
The commutative diagram (\ref{exact}) induces by functoriality
 \begin{equation}\label{exaseq}
\vcenter{\xymatrix{
\pi_i(\tP_{k+1})\ar[d]^{} \ar[r]^{({{\jmath}}_{k+1})_*}&
\pi_i(\Omega(\tP_k))\ar[d]^{}\ar[r]^{\simeq} &
 \pi_{i+1}(\tP_k)\ar[d]
 \\
\pi_i(\tilde{P}_{k+1})\ar[r]^{(\tilde{\jmath}_{k+1})_*} &
\pi_i(\Omega(\tilde{P}_k))\ar[r]^{\simeq} &
 \pi_{i+1}(\tilde{P}_k)
}}
\end{equation}
 for all $k\in \{0,1,\ldots, 7\}$. Recall that  $P_0=\SO_{16n}$, $\tilde{P}_0=\U_{16n}$,
 and both $(\jmath_{k+1})_*$ and $(\tilde{\jmath}_{k+1})_*$ are
 isomorphisms for any $i$ which is sufficiently small  compared to $n$
 (see Subsection \ref{proofbo} and the references therein).
Since $\pi_{i+8}(\SO_{16n})=\pi_{i+8}(\O_{16n})$, we obtain  the diagram:
\begin{equation*}
\vcenter{\xymatrix{
\pi_{i+8}(\O_{16n})\ar[d]^{f_{16n}^{i+8}} \ar[r]^{\simeq }&
\pi_{i}(\tP_8)\ar[d]^{} &
\\
\pi_{i+8}(\U_{16n})\ar[r]^{\simeq} &
\pi_{i}(\tilde{P}_8) &\\
}}
\end{equation*}
 Finally, from Theorem \ref{co} we deduce that we have a commutative diagram
of the form
\begin{equation*}
\vcenter{\xymatrix{
\pi_{i}(\tP_8)\ar[d]^{} \ar[r]^{\simeq }&
\pi_{i}(\SO_n)\ar[d]^{f^n_i} &
\\
\pi_{i}(\tilde{P}_8)\ar[r]^{\simeq} &
\pi_{i}(\U_n) &\\
}}
\end{equation*}
We only need to use the fact that $\pi_i(\SO_n)=\pi_i(\O_n)$.

We now analyze the case $i=0$. We have  $\pi_0(\U) =\pi_8(\U)=\{0\}$,
thus the maps $f_0$ and $f_8$ are clearly equal.
This finishes the proof of the theorem.
\end{proof}

\begin{table}[h]
\begin{tabular}{|c| cccccccc |}
	\hline
$i$ mod 8& 0 &  1& 2 & 3 & 4 & 5 & 6 & 7     \\
	\hline
$\pi_i(\O)$& $\bZ_2$ & $\bZ_2$ & 0 & $\bZ$& 0 & 0 & 0 & $\bZ$   \\
$  \pi_i (\U)$ & 0 & $\bZ$ & 0 & $\bZ$ & 0 & $\bZ$& 0 & $\bZ$   \\
$f_i$ & $0$ & $0$ & $0$ & $k\mapsto 2k$ & 0 & 0 & 0 & id \\
\hline
\end{tabular}
\caption{}
\end{table}
To calculate the maps $f_i$ explicitly,
we can  use the long exact homotopy sequence of the principal bundle $\O_m \to
\U_m \to \U_m/\O_m$. This information is described in Table 1
(where we have used the table from \cite[p.~142]{Mi}).

Justifications are needed only for the maps $f_3$ and $f_7$.
Let us  calculate the map
$f_3: \pi_3(\O)\to \pi_3(\U)$. Since $\pi_4(\U/\O)= 0$
and $\pi_3(\U/\O)= \bZ_2$ (cf. e.g. \cite[Section 1]{Bo}), we obtain the following exact sequence:
$$0 \to \bZ \stackrel{f_3}{\to} \bZ \to \bZ_2 \to 0.$$
This implies the desired description of $f_3$.
As about $f_7$, the relevant exact sequence is
$$0\to \bZ \stackrel{f_7}{\to} \bZ  \to 0.$$

\begin{remark}\label{notes} {\rm
Note that the  two exact sequences above can be used to show that for any
$j=0,1,2, \ldots$, the map $f_{8j+3} : \bZ \to \bZ$ is given by
$k\mapsto 2k$, $k\in \bZ$, and  $f_{8j+7}:\bZ \to \bZ$ is the identity map.
Therefore this simple argument gives an alternative proof to Theorem \ref{maintheo}.
}
\end{remark}

We can combine Theorem \ref{maintheo} above with the commutative diagram
given by (\ref{exaseq}) and the results concerning the exact expressions of the embeddings
$P_k\subset \tilde{P}_k$, $k=1,2,\ldots, 8$ obtained in
Appendix  \ref{apa} (see Table 5  and Subsections \ref{2e} - \ref{8e}). We deduce:

\begin{corollary}\label{cor} Let $A_m\hookrightarrow B_m$ be given by any of the inclusions
$$\begin{array}{lllll}
\O_{2m}/\U_{m}\subset \G_{m}(\bC^{2m}), &
\U_{2m}/\Sp_m\subset \U_{2m},&
\G_m(\bH^{2m})\subset \G_{2m}(\bC^{4m}),&
\Sp_m\subset \U_{2m},\\
\Sp_m/\U_m\subset \G_m(\bC^{2m}),&
\U_m/\O_m\subset \U_m,&
\G_m(\bR^{2m})\subset \G_m(\bC^{2m}).& {}
\end{array}
$$
Then the maps $\pi_i(A_m)\to \pi_i(B_m)$ induced between the stable homotopy groups are
stable relative to $m$ and periodic relative to $i$, with period equal to 8.
\end{corollary}

The exact expression of the stable maps $\pi_i(A_m)\to \pi_i(B_m)$ can be deduced from the table  above  by finding $n$ and $k$ such that 
$P_k$ and $\tilde{P}_k$ are equal to $A_m$ respectively $B_m$ for a certain
$m$ which depends on $n$ (see Table 5 for $1\le k\le 7$). 
The only embedding for which this is not possible is 
$\O_{2m}/\U_{m}\subset \G_{m}(\bC^{2m})$. In this case,
we note that $P_1=\SO_{2m}/\U_m$, where $m=8n$ (see Subsection
\ref{sochain} or Table 5). Consequently,  $\pi_i(\O_{2m}/\U_{m})=\pi_i(P_1)$ for any
$i\neq 0$ and therefore in this case the map  $\pi_i(\O_{2m}/\U_{m})\to \pi_i(\G_m(\bC^{2m}))$ is equivalent  to  $\pi_i(P_1)\to \pi_i(\tilde{P}_1)$ in the sense of Definition \ref{letapr}. For $i\equiv 0$ mod 8, we note that
$\pi_i(\G_m(\bC^{2m}))=\{0\}$, hence the map $\pi_i(\O_{2m}/\U_{m})\to \pi_i(\G_m(\bC^{2m}))$ is identically zero.
To deal with any of the remaining six inclusions we just take  
 $k\in \{2,3, \ldots, 7\}$ and use inductively the commutative diagram (\ref{exaseq})
to deduce that the map $\pi_i(P_k)\to \pi_i(\tilde{P}_k)$ is equivalent to
$\pi_{i+k}(\O_m) \to \pi_{i+k}(\U_m)$ (here $m=16n$ is in the stability range).
For instance the stable maps between homotopy groups induced by the inclusion
$\Sp_m\subset \U_{2m}$ are described in Table 2 (see also Remark
\ref{samesp}).

\begin{table}[h]
\begin{tabular}{|c| cccccccc |}
	\hline
$i$ mod 8& 0 &  1& 2 & 3 & 4 & 5 & 6 & 7     \\
	\hline
	$  \pi_i (\Sp)$ & 0 & 0 & 0 & $\bZ$ & $\bZ_2$ & $\bZ_2$& 0 & $\bZ$    \\
	$\pi_i(\U)$& 0 & $\bZ$ & 0 & $\bZ$ & 0 & $\bZ$& 0 & $\bZ$   \\
$ \pi_i(\Sp) \to \pi_i(\U)$ & $0$ & $0$  & 0 & id  & 0 & 0 & 0 & $k\mapsto 2k$  \\
\hline
\end{tabular}
\caption{}
\end{table}

\subsection{The maps induced by  $\U_m\hookrightarrow \Sp_m$ }\label{lasts}

In the same way as in the previous subsection, we consider the inclusion
map $\U_m\to \Sp_m$ and the maps $g_i^m:\pi_i(\U_m)\to \pi_i(\Sp_m)$
induced between homotopy groups. As in Lemma \ref{equivc}, if we fix $i$ and take any $m$
which is sufficiently larger than $i$, all of these group homomorphisms 
are equivalent in the sense of Definition \ref{letapr}.
Denote by $g_i$ the equivalence class of these maps.
The following result can be proved with the same methods
as Theorem \ref{maintheo}.

\begin{theorem}\label{maitheo} We have $g_{i+8}=g_i$.
\end{theorem}

Table 3  describes the maps $g_i$ explicitly.

\begin{table}[h]
\begin{tabular}{|c| cccccccc |}
	\hline
$i$ mod 8& 0 &  1& 2 & 3 & 4 & 5 & 6 & 7     \\
	\hline
$\pi_i(\U)$& 0 & $\bZ$ & 0 & $\bZ$ & 0 & $\bZ$& 0 & $\bZ$   \\
$  \pi_i (\Sp)$ & 0 & 0 & 0 & $\bZ$ & $\bZ_2$ & $\bZ_2$& 0 & $\bZ$    \\
$ g_i$ & $0$ & $0$ & $0$ & $k\mapsto 2k$ & 0 & $k \mapsto k \ {\rm mod \ } 2$ & 0 & id \\
\hline
\end{tabular}
\caption{}
\end{table}

\begin{remark}\label{remer} {\rm 
The results in Table 3 have been obtained as direct consequences of  the long exact homotopy sequence of the principal bundle $\U_m \to \Sp_m \to \Sp_m/\U_m$ and
the knowledge of $\pi_i(\Sp/\U)$, $i=0,1,2,\ldots$.
In fact, this long exact sequence can also be used to give an alternative proof of the periodicity of the maps $\pi_i(\U)\to \pi_i(\Sp)$. }
 \end{remark}

 In the same way as Corollary \ref{cor}, we can prove the following result (this time using Table 6 in Appendix \ref{apa} and Subsections \ref{2em} - \ref{8em}).
 
 \begin{corollary}\label{coragain} Let  $A_m\hookrightarrow B_m$ be given by any of the inclusions
 $$\begin{array}{lllll}
 \G_m(\bC^{2m}) \subset \Sp_{2m}/\U_{2m},&
 \U_m\subset \U_{2m}/\O_m,&
\G_m(\bC^{2m})\subset \G_{2m}(\bR^{4m}),&\U_m\subset \O_{2m},
\\
\G_m(\bC^{2m})\subset \O_{4m}/\U_{2m},&
\U_m\subset \U_{2m}/\Sp_m,&
\G_m(\bC^{2m})\subset \G_m(\bH^{2m}).&{}
\end{array}
$$
Then the maps $\pi_i(A_m)\to \pi_i(B_m)$ induced between the stable homotopy groups are
stable relative to $m$ and periodic relative to $i$, with  period equal to 8.
\end{corollary}

These maps $\pi_i(A_m)\to \pi_i(B_m)$ mentioned above can be described explicitly, by using Table 3 and the fact that
$\pi_i(\tilde{P}_k)\to \pi_i(\bar{P}_k)$ is equivalent to
$\pi_{i+k}(\U_m) \to \pi_{i+k}(\Sp_m)$.
 For example, the stable maps $\pi_i(\U_m)\to \pi_i(\O_{2m})$  are described in Table 4.
 \begin{center}
\begin{table}[h]
\begin{tabular}{|c| cccccccc |}
	\hline
$i$ mod 8& 0 &  1& 2 & 3 & 4 & 5 & 6 & 7     \\
	\hline
	$\pi_i(\U)$& 0 & $\bZ$ & 0 & $\bZ$ & 0 & $\bZ$& 0 & $\bZ$   \\
	$\pi_i(\O)$& $\bZ_2$ & $\bZ_2$ & 0 & $\bZ$& 0 & 0 & 0 & $\bZ$   \\
$ \pi_i(\U) \to \pi_i(\O)$ & $0$ & $k\mapsto k$ mod $2$ & $0$ & id & 0 & 0 & 0 & $k\mapsto 2k$  \\
\hline
\end{tabular}
\caption{}
\end{table}
\end{center}

\appendix

\section{Standard inclusions of symmetric spaces}\label{apa}

Explicit descriptions of the spaces in the $\SO$-Bott chain have been obtained by
Milnor in \cite[Section 24]{Mi} using orthogonal complex structures of $\bR^{16n}$,
i.e.~elements $J$ of $\O_{16n}$ with the property that $J^2=-I$.
 A similar construction works for any matrix Lie group, as it has been pointed out in \cite{Qu}. For our needs, we describe the $\U$-Bott chain in these terms. We start with the following definition.

\begin{definition}\label{wat} An element $J\in \U_{16n}$ is a  {\rm complex structure} if  $J^2=-I$.
\end{definition}

Like in the case of  the orthogonal group (see \cite[Lemma 24.1]{Mi}) we can   identify
complex structures in $\U_{16n}$ with midpoints of shortest  geodesic segments in $\U_{16n}$ from $I$ to $-I$.
More specifically, recall from Subsection \ref{poled} that $-I$ is a pole of $(\U_{16n}, I)$ and
 the space of shortest geodesic segments in $\U_{16n}$ from
$I$ to $-I$ is the union of all conjugacy orbits $\U_{16n}.\gamma_k|_{[0,1]}$, $0\le k \le 16n$ (see  Equation (\ref{gamak})).

\begin{lemma}\label{theset}  The set of all midpoints of the geodesic segments in the union 
$\bigcup_{0\le k \le 16n} \U_{16n}.\gamma_k|_{[0,1]}$ 
coincides with the set of all complex structures in $\U_{2q}$.
\end{lemma}

\begin{proof} 
Let $\gamma|_{[0,1]}: [0,1]\to \U_{16n}$ be a   geodesic segment in the union above: it satisfies $\gamma(0)=I$ and $\gamma(1)=-I$.
Then $\gamma :  \bR \to \U_{16n}$ is a one-parameter subgroup.
Thus we have
$$\gamma\left(\frac{1}{2}\right)^2=\gamma(1)=-I.$$
This means that $\gamma\left(\frac{1}{2}\right)$ is a complex structure.
To prove the converse inclusion, take $J\in \U_{16n}$ such that $J^2=-I$.
Then the eigenvalues of $J$ are $\pm i$, hence $J$ is $\U_{16n}$-conjugate to
a matrix of the form
$$\left(%
\begin{array}{ccccccc}
    i I_k& 0\\
0 & - i I_{16n-k}
\end{array}%
\right)
 $$
 for some $k\in \{0,1,\ldots, 16n\}$. The converse inclusion is proved.
\end{proof}

Recall from Subsection \ref{ubott} that $\tilde{P}_1$ is the top-dimensional
 s-centriole of $(\U_{16n}, I)$ and $J_1$ is an element of $\tilde{P}_1$.
The previous lemma says that the union of all  s-centrioles in $\U_{16n}$ from
$I$ to $-I$ is the same as the set of all complex structures in $\U_{16n}$.
We deduce that
\begin{equation} \label{was} \tilde{P}_1=\{J \in \U_{16n} \ : \ J^2=-I\}_{J_1},
\end{equation}
where we have used the notation established at the beginning of Subsection \ref{fixedpoints}.

 We saw afterwards that we have the isometry $\tilde{P}_1\simeq \G_{8n}(\bC^{16n})$,
where $\tilde{P}_1$ is equipped with the metric induced by its embedding
in $\U_{16n}$ and $\G_{8n}(\bC^{16n})$ with  the usual symmetric space metric.
We   defined $\tilde{P}_2$ as the (unique)  s-centriole of
$(\tilde{P}_1,J_1)$ relative to $-J_1$ and then we  fixed an element $J_2$ of $\tilde{P}_2$.

\begin{lemma}\label{commute}
The centrosome $C_{-J_1}(\tilde{P}_1,J_1)$ can be expressed as:
\begin{align}C_{-J_1}(\tilde{P}_1,J_1) =\{J \in \tilde{P}_1  \ : \   JJ_1=-J_1J\}.\label{jjj}
\end{align}
Consequently,
$\tilde{P}_2=\{J \in \tilde{P}_1  \ : \   JJ_1=-J_1J\}_{J_2}.$
\end{lemma}

\begin{proof}
We first take  $J=\gamma(\frac{1}{2})$, where $\gamma: \bR \to \tilde{P}_1$
is a geodesic  such that
$$\gamma(0)=J_1 \quad {\rm and} \quad \ \gamma(1)=-J_1.$$
But $\tilde{P}_1$ is a totally geodesic submanifold of $\U_{16n}$, thus
$\gamma$ is a geodesic in $\U_{16n}$. We deduce that there exists
$x$ in $\u_{16n}$
such that
$$\gamma(t)=J_1\exp(2tx), \ t\in \bR.$$
The condition $\gamma\left(\frac{1}{2}\right)=J$ implies that
$J_1\exp(x)=J.$ Multiplying from the left by $J_1$  and
taking into account that $J_1^2=-I$ gives
$\exp(x)=-J_1J.$
Consequently, we have
$$-J_1=\gamma(1)=J_1\exp(x)\exp(x)
= J(-J_1J)=-JJ_1J.$$
This implies that $JJ_1=-J_1J$.

We now prove the converse inclusion.
Take $J\in \tilde{P}_1$ such that $JJ_1=-J_1J$.
Let $\gamma : \bR \to \tilde{P}_1$ be a geodesic with the property that
$$\gamma(0)= J_1 \quad {\rm and} \quad \gamma\left(\frac{1}{2}\right)=J.$$

\noindent {\it Claim.} $\gamma(1)=-J_1$.

Indeed, the curve $J_1^{-1}\gamma$ is a geodesic in $\U_{16n}$, thus a
one-parameter group. In other words, we have
$$J_1^{-1}\gamma(t)=\exp(2tx), \ t\in \bR,$$
where $x\in \u_{16n}$.
This implies that
$J=\gamma(\frac{1}{2}) = J_1 \exp(x)$, hence
$\exp(x)=J_1^{-1}J=-J_1J=JJ_1,$
and consequently
$$\gamma(1)= J_1\exp(2x)=J_1\exp(x)\exp(x)
= J(JJ_1)=-J_1.$$
\end{proof}

Note that the previous two results are special cases of
\cite[Lemmata 3.1 and 3.2]{Qu}.

We have the isometry $\tilde{P}_2 \simeq \U_{8n}$, where $\tilde{P}_2$ has the
submanifold metric and $\U_{8n}$ the bi-invariant metric described by Equation
(\ref{xyz}) with $q=4n$ (see Subsection \ref{poleun}).
The pair $(\tilde{P}_2, J_2)$ has exactly one
pole, which is $-J_2$.
We defined $\tilde{P}_3$ as the top-dimensional  s-centriole of
$(\tilde{P}_2,J_2)$ relative to $-J_2$ and we  fixed $J_3\in \tilde{P}_3$.
Since $\tilde{P}_2$ is a totally geodesic submanifold of $\U_{16n}$,
we can use  the same reasoning as in the proof of Lemma \ref{commute}
to show that
$$\tilde{P}_3=\{J \in \tilde{P}_2 \ : \ JJ_2=-J_2J\}_{J_3}.$$
In the same way, for any $k\in \{2, \ldots, 8\}$ we have
$$\tilde{P}_k=\{J\in \tilde{P}_{k-1} \ : \ JJ_{k-1}=-J_{k-1}J\}_{J_k}.$$

These new presentations of the spaces $\tilde{P}_1, \ldots, \tilde{P}_8$, along with those obtained by Milnor in  \cite[Section 24]{Mi} for the spaces  $P_1, \ldots, P_8$  lead us to descriptions of the embeddings between Bott chains which are
discussed in Section \ref{inclubot}. Concretely, they are given by Tables
5 and 6  together with the list \ref{1e} - \ref{8em}.

\begin{table}[h]\label{tal1}
\begin{tabular}{|l| ll|l |}
	\hline
	$k$ & $P_k$ & $\tilde{P}_k$ &$P_k\subset \tilde{P}_k$\\	\hline
		\hline
0& $\SO_{16n}$&$\U_{16n}$  &  \ref{1e}  composed with $\SO_{16n}\subset \O_{16n}$   \\
	\hline
1& $\SO_{16n}/\U_{8n}$&$\G_{8n}(\bC^{16n})$  &  \ref{2e} composed with $\SO_{16n}/\U_{8n}\subset  \O_{16n}/\U_{8n}$   \\
	\hline
2& $\U_{8n}/\Sp_{4n}$&$\U_{8n}$  &  \ref{3e}     \\
	\hline
3& $\G_{2n}(\bH^{4n})$&$\G_{4n}(\bC^{8n})$  &  \ref{4e}     \\
	\hline	
4& $\Sp_{2n}$&$\U_{4n}$  &  \ref{5e}     \\
	\hline	
5& $\Sp_{2n}/\U_{2n}$&$\G_{2n}(\bC^{4n})$  &  \ref{6e}     \\
	\hline	
6& $\U_{2n}/\O_{2n}$&$\U_{2n}$  &  \ref{7e}     \\
	\hline	
7& $\G_n(\bR^{2n})$&$\G_{n}(\bC^{2n})$  &  \ref{8e}     \\
\hline
8& $\SO_{n}$&$\U_{n}$  &  \ref{1e}  composed with $\SO_{n}\subset \O_{n}$   \\
	\hline	
\end{tabular} 
\caption{}
\end{table}
\begin{table}[h]\label{tal2}
\begin{tabular}{|l| ll|l |}
	\hline
	$k$ & $\tilde{P}_k$ & $\bar{P}_k$ &$\tilde{P}_k\subset \bar{P}_k$\\	\hline
		\hline
0& $\U_{16n}$&$\Sp_{16n}$  &  \ref{1em}     \\
	\hline
1& $\G_{8n}(\bC^{16n})$&$\Sp_{16n}/\U_{16n}$  &  \ref{2em}     \\
	\hline
2& $\U_{8n}$&$\U_{16n}/\O_{16n}$  &  \ref{3em}     \\
	\hline
3& $\G_{4n}(\bC^{8n})$&$\G_{8n}(\bR^{16n})$  &  \ref{4em}     \\
	\hline	
4& $ \U_{4n}$&$\SO_{8n}$  &  \ref{5em}     \\
	\hline	
5& $\G_{2n}(\bC^{4n})$&$\SO_{8n}/\U_{4n}$  &  \ref{6em}     \\
	\hline	
6& $\U_{2n}$&$\U_{4n}/\Sp_{2n}$  &  \ref{7em}     \\
	\hline	
7& $\G_n(\bC^{2n})$&$\G_n(\bH^{2n})$  &  \ref{8em}     \\
		\hline	
8& $\U_{n}$&$\Sp_{n}$  &  \ref{1em}     \\
	\hline		
\end{tabular}
\caption{}
\end{table}

In what follows we will be frequently using, without pointing it out each time,  presentations of  certain classical symmetric  spaces as  given in   \cite[Section 24]{Mi} (see also \cite[Section 19]{Es}).

\subsection{The inclusion $\O_r\subset \U_r$}\label{1e}
To any orthogonal isomorphism $A:\bR^r\to \bR^r$ one attaches its complex-linear extension $A^c:\bC^r\to \bC^r$, which is defined by
$A^c(u+iv):=A(u)+iA(v)$, for all $u,v\in \bR^r$.
One can see that ${A}^c$ preserves the norm of a vector in $\bC^r$ with respect to the
canonical Hermitian product. Thus, $A^c$ lies in $\U_{r}$.

\subsection{The inclusion $\O_{2r}/\U_r\subset \G_r(\bC^{2r})$}\label{2e}
The quotient $\O_{2r}/\U_r$ is identified with the space of  all orthogonal complex structures of $\bR^{2r}$, that is, of
all $J\in \O_{2r}$ with the property that $J^2=-I$. The inclusion
 $\O_{2r}/\U_r\subset \G_r(\bC^{2r})$ assigns to any such $J$ the eigenspace
 $E_i(J^c)=\{v\in \bC^{2r} \ : \ J^c(v)=iv\}$.

 \subsection{The inclusion $\U_{2r}/\Sp_r \subset \U_{2r}$}\label{3e}
 Fix $J_0\in \O_{4r}$ an orthogonal complex structure of $\bR^{4r}$.
 Let $J_0^c: \bC^{4r}\to \bC^{4r}$ be its complex linear extension.
The eigenspaces $V^+:=E_i(J_0^c)$ and $V^-:=E_{-i}(J_0^c)$ 
 are complex vector subspaces of $\bC^{4r}$ of dimension equal to $2r$, since the complex conjugation is an   ($\bR$-linear) isomorphism between $V^+$ and $V^-$. 
The quotient
$\U_{2r}/\Sp_r$ can be identified with the space of all orthogonal complex structures
$J$ of $\bR^{4r}$ that anticommute with $J_0$.
The complex linear extension $J^c$ of such a $J$ maps $V^+$ to 
$V^-$, being obviously a unitary isomorphism.
The inclusion $\U_{2r}/\Sp_r\hookrightarrow \U_{2r}$ assigns to  $J$ the map
$J^c|_{V^+}:V^+\to V^-$, where both $V^+$ and $V^-$ are identified
with $\bC^{2r}$.

\subsection{The inclusion $\G_r(\bH^{2r})\subset \G_{2r}(\bC^{4r})$}\label{4e}
We first identify $\bC$ with the subspace of  $\bH$ consisting of all
quaternions $a+bi+cj+dk$ with $c=d=0$.
This allows us to equip $\bH^{2r}$ with the structure of complex vector space induced by multiplication with complex numbers from the right.  It also allows us to embed
$\bC^{2r}$ into $\bH^{2r}$. 
In this way we obtain the following identification of complex vector spaces: $\bH^{2r}=\bC^{2r}\oplus j \bC^{2r} =\bC^{4r}$.
The Grassmannian $\G_r(\bH^{2r})$ consists of all  right $\bH$-submodules 
of $\bH^{2r}$ of dimension equal to $r$. 
The map
$\G_r(\bH^{2r})\hookrightarrow \G_{2r}(\bC^{4r})$ attaches to any such  submodule 
$V\subset \bH^{2r}$ the space $V$ itself, regarded as a
$2r$-dimensional complex subspace of $\bC^{4r}$.

\subsection{The inclusion $\Sp_r\subset \U_{2r}$}\label{5e}
As explained before, we can regard $\bH^r=\bC^r \oplus j \bC^r=\bC^{2r}$
as a complex vector spaces. The map $\Sp_r\hookrightarrow \U_{2r}$
assigns to any symplectic ($\bH$-linear on the right) isomorphism $A:\bH^r\to \bH^r$
  the map $A$ itself, regarded as a unitary ($\bC$-linear) isomorphism $\bC^{2r}\to \bC^{2r}$.
 A description of this embedding in matrix form can be found for instance
\cite[Ch.~I, Section 1.11]{Br-tD} (see also the beginning of Section \ref{inclubot}).

  \subsection{The inclusion $\Sp_r/\U_r\subset \G_r(\bC^{2r})$}\label{6e}
  The quotient $\Sp_r/\U_r$ can be identified with the space of all complex forms of the quaternionic space $\bH^r$, that is, all $V\subset \bH^r$ which is a complex vector subspace relative to the identification
  $\bH^r=\bC^r+j\bC^r=\bC^{2r}$ mentioned above and satisfies $\bH^r=V\oplus jV$.
  The inclusion map $\Sp_r/\U_r\hookrightarrow \G_r(\bC^{2r})$ assigns to any
  such $V$ the space $V$ itself.

  \subsection{The inclusion $\U_r/\O_r\subset \U_r$}\label{7e}
The quotient $\U_r/\O_r$ can be identified with the space of all real forms of $\bC^r$, that is, all
real vector subspaces $V\subset \bC^r$ such that $\bC^r=V\oplus iV$.
This can be further identified with the space of all orthogonal ($\bR$-linear) automorphisms
of $\bC^r=\bR^{2r}$ that are anti-complex linear and square to $I$:
the identification is given by attaching to    such an automorphism
its 1-eigenspace.
If we fix an anti-complex linear orthogonal automorphism $B_0$ of $\bR^{2n}$, then
$$\U_r/\O_r=\{B_0A \ : \ A\in \U_r, (B_0A)^2=I\}.$$
The inclusion map $\U_r/\O_r \to \U_r$ maps $B_0A$ to $A$.

    \subsection{The inclusion $\G_r(\bR^{2r}) \subset \G_r(\bC^{2r})$}\label{8e}
  This map assigns to any $r$-dimensional real vector subspace of $\bR^{2r}$ the space $V\otimes \bC$, which is an $r$-dimensional complex vector subspace of $\bC^{2r}$.

\subsection{The inclusion $\U_r\subset \Sp_r$}\label{1em}
Let $R_i$ and $R_j$ be the maps $\bH^r \to \bH^r$ given by multiplication from the right by the quaternionic units $i$ and $j$.
Then $\Sp_{r}$ can be characterized as the space of all $\bR$-linear endomorphisms of $\bH^r$ which commute with $R_i$ and $R_j$ and preserve the norm of any vector
in $\bH^r$ relative to the canonical inner product
of $\bH^{r}$.  Let us consider the splitting
$\bH^{r} = \bC^{r}\oplus j \bC^{r} .$
The group $\U_{r}$ consists of all $\bR$-linear endomorphisms of
$\bC^{r}$ which commute with $R_i$ and  preserve the norm of any vector
in $\bC^r$ relative to the canonical Hermitian product of $\bC^{r}$.
The desired embedding $\U_{r} \hookrightarrow \Sp_{r}$ is given by
\begin{equation*}\U_{r}\ni A \mapsto A^h\in \Sp_{r},\end{equation*}
where  $A^h: \bH^{r} \to \bH^{r}$ is
determined by:
$$A^h(v+jw):=Av+j(\bar{A}w), \quad v,w\in \bC^r.$$
(One can easily verify that ${A}^h$ lies in $\Sp_{r}$.)

\subsection{The inclusion  $\G_r(\bC^{2r})\subset \Sp_{2r}/\U_{2r}$}\label{2em}
The quotient $\Sp_{2r}/\U_{2r}$ can be identified with the space of all complex forms of $\bH^{2r}$,
that is,
of all  real vector subspace $X \subset \bH^{2r}$ with the  property that
 $R_iX=X$, i.e. $X$ is a complex vector subspace of $\bH^{2r}$,
and  $\bH^{2r}=X \oplus R_jX$ (orthogonal direct sum).
The inclusion $\G_r(\bC^{2r})\hookrightarrow \Sp_{2r}/\U_{2r}$ assigns to the $r$-dimensional complex vector subspace $V \subset \bC^{2r}$  the space
$V\oplus R_j V^\perp$, where $V^\perp$ is the orthogonal complement of $V$ in
$\bC^{2r}$.

\subsection{The inclusion $\U_{r}\subset \U_{2r}/\O_{2r}$}\label{3em}
Recall that $\U_{2r}/\O_{2r}$ is the space of all real forms of $\bC^{2r}$
(see Subsection \ref{7e}). Also recall that $\bH^r$ is a complex vector space
relative to multiplication by complex numbers from the right, the dimension being equal to $2r$.
Let us now consider  the splitting
$\bH^{r}=\bC^r\oplus R_j \bC^r$.
 The inclusion $\U_{r}\hookrightarrow \U_{2r}/\O_{2r}$ can be described as follows:
$$\U_r\ni A\mapsto V:=\{ v+R_j Av \ : \ v\in \bC^r\}.$$
 Note that $V$  described by this equation is a real form  of $\bH^r$, where the latter is a complex vector space in the way mentioned above. Indeed, this follows
readily from the fact that $Vi=\{ v -R_jAv \ : \ v\in \bC^r\}.$

 \subsection{The inclusion $\G_r(\bC^{2r})\subset\G_{2r}(\bR^{4r})$}\label{4em}
This map
assigns to a complex $n$-dimensional vector subspace
$V\subset \bC^{2r}$ the space $V$ itself, viewed as a real vector subspace of
$\bC^{2r}=\bR^{4r}$.

\subsection{The inclusion $\U_{r}\subset \SO_{2r}$.}\label{5em}
We identify  $\bC^r=\bR^r\oplus i\bR^r$ with $\bR^{2r}$ and  make  the following elementary  observations:
 a $\bC$-linear transformation of $\bC^{r}$ is also $\bR$-linear;
 the norm of a vector in $\bC^{r}$ relative to the standard Hermitian inner product 
 is equal to its norm in $\bR^{2r}$ relative to the standard Euclidean inner product.
We are lead to the subgroup embedding $\U_r \hookrightarrow \O_{2r}$. Since
 $\U_r$ is connected, we actually get $\U_r  \hookrightarrow \SO_{2r}$.

\subsection{The inclusion $\G_r(\bC^{2r})\subset \SO_{4r}/\U_{2r}$}\label{6em}
We start with the embedding $\U_{2r}\subset \SO_{4r}$ described in Subsection
\ref{5em}.
It induces the inclusion $\{J\in \U_{2r} \ : \ J^2=-I\}\subset \{J\in \SO_{4r} \ : \ J^2=-I\}$.
The first space can be identified with  the Grassmannian of all complex vector subspaces in
$\bC^{2r}$ (see Lemma \ref{theset} and Subsection \ref{poled}). Among its connected components we can find $\G_{r}(\bC^{2r})$.
This is contained in one of the two connected components of
$ \{J\in \SO_{4r} \ : \ J^2=-I\}$. They are both diffeomorphic to $\SO_{4r}/\U_{2r}$.
The desired embedding is now clear.

\subsection{The inclusion $\U_{r}\subset \U_{2r}/\Sp_{r}$}\label{7em}
We first consider $$A_r:=
\left(%
\begin{array}{ccccccc}
iI & 0 \\
0 & -iI \end{array}%
\right)\in \U_{2r},$$
which is an orthogonal complex structure of $\bR^{4r}$ via the embedding
described at \ref{5em}.
 Denote by $\U(\bR^{4r}, A_r)$ the set of all elements of $\O_{4r}$ which
  commute with $A_r$.
  This is a subgroup of $\O_{4r}$ which is isomorphic to
  $\U_{2r}$. It  acts transitively, via group conjugation, on the set of all
  $J\in \O_{4r}$ with
  $J^2=-I$ and $A_rJ=-JA_r$. Moreover, 
  the stabilizer of any $J$ is isomorphic to $\Sp_{r}$. In this way we obtain 
 the identification
 $$\{J \in \O_{4r} \ : \ J^2=-I, JA_r=-A_rJ\} = 
 \U_{2r}/\Sp_r.$$ 
 The embedding $\U_r\hookrightarrow \U_{2r}/\Sp_r$ assigns to an arbitrary $X\in \U_{r}$ 
 the matrix
$$ A:=\left(%
\begin{array}{ccccccc}
0 & -X^{-1} \\
X & 0 \end{array}%
\right)\in \U_{2r},
$$
which is regarded as an element of $\O_{4r}$ in the same way as before,
i.e.~by using the embedding \ref{5em}. (One can easily verify that ${A}^2=-I$ and 
$A_r{A}=-{A}A_r$.)

\subsection{The inclusion $\G_r(\bC^{2r}) \subset \G_r(\bH^{2r})$}\label{8em} 
We consider again the embedding $\bC^{2r}\subset \bH^{2r}$ defined in Subsection \ref{4e}.
The embedding $\G_r(\bC^{2r}) \hookrightarrow \G_r(\bH^{2r})$ assigns to  a complex 
$r$-dimensional vector subspace $V\subset \bC^{2r}$ the space
$V\otimes_\bC\bH =\{v+wj \ : \ v,w\in V\},$
  which is an $\bH$-linear subspace of $\bH^{2r}$ of dimension $r$.

\section{The isometry types of $P_4$ and  $P_8$}\label{istype}
For any $r\ge 1$, we consider the {\it standard}  bi-invariant Riemannian metric on
 each of the groups $\SO_r$, $\U_r$, and $\Sp_r$. By definition, they are   given by
 $\langle X, Y\rangle =-{\rm tr}(XY)$,
for any $X,Y$ in the Lie algebra of $\SO_r$, respectively $\U_r$;
as about  $\Sp_r$,  the metric is induced  by  its canonical embedding in $\U_{2r}$, where the latter group is equipped with the  standard metric divided by two
(see also the beginning of Section \ref{inclubot}).

The $\SO$-Bott chain $P_0, P_1, \ldots, P_8$ has been defined in Subsection
\ref{sochain}.
Recall that  $P_1, \ldots, P_8$
are totally geodesic submanifolds of $P_0=\SO_{16n}$, the latter space being equipped with  the standard metric.
The main goal of  this section is to prove the following result.

\begin{proposition}\label{isotype}
 (a) If we equip $P_4$ with the submanifold metric, then  $P_4$ is
 isometric to
    $ \Sp_{2n}$, where the metric on the latter space
    is   eight times the standard one.

 (b) If we equip $P_8$ with the submanifold metric, then $P_8$ is isometric to
 $ \SO_{n}$, where the metric on the latter space
    is  sixteen times the standard one.
\end{proposition}

\begin{proof}
  The proof will be divided into two steps. In the first step, we show that the results stated by the proposition hold true for a particular choice of  
  complex structures $J_1, \ldots, J_7$. (Afterwards we will address the general situation.)
Concretely, we write
$$\bH^{4n}=\bR^{4n}\oplus \bR^{4n}i \oplus \bR^{4n}j\oplus \bR^{4n}k,$$
and identify in this way
$\bH^{4n}=\bR^{16n}$. Take $J_1:=R_i$ and $J_2:=R_j$, that is, multiplication on
$\bH^{4n}$
 from the right by the
quaternionic units $i$, respectively $j$. We take $J_3$ in such a way, that
$J_1J_2J_3$ is given by
$$J_1J_2J_3(q_1, \ldots, q_{4n}):=(q_1, \ldots, q_{2n}, -q_{2n+1},
\ldots, -q_{4n}),$$
for all $(q_1,\ldots, q_{4n})\in \bH^{4n}$.

We now perform Milnor's construction of the space $P_4$ (cf.~\cite[p.~139]{Mi}, see also Subsection \ref{sochain} above).
First, note that the eigenspace decomposition of $J_1J_2J_3: \bH^{4n} \to \bH^{4n}$
is $\bH^{4n}=\bH^{2n}\oplus (\bH^{2n})^\perp$, where
$\bH^{2n}$ stands here for the space of all vectors in $\bH^{4n}$ with
the last $2n$ entries equal to 0 and $(\bH^{2n})^\perp$ is
the space of all vectors in $\bH^{4n}$ with
the first $2n$ entries equal to 0.
The space $P_4$ consists of all $J\in P_3$ which anticommute with
 $J_3$. If $J$ is such a transformation, then  $J_3J$
maps $\bH^{2n}$ to $(\bH^{2n})^\perp$ as a  $\bH$-linear map (relative to scalar multiplication from the right)  that preserves
the norm of any vector.

Let us now consider the subgroup of $\SO_{16n}$ which consists of  all  $\bR$-linear endomorphisms
of $\bR^{16n}$ that are  $\bH$-linear, i.e.~commute with $J_1$ and $J_2$,
and preserve the norm of a vector. This group is just $\Sp_{4n}$.
We prefer to see its elements as $4n\times 4n$ matrices,
say $A$, with entries in $\bH$, such that $AA^*=I_{4n}$.
From the above observation,  $J_3P_4:=\{J_3J\ : \ J\in P_4\}$ is the same as
the space of all elements of $\Sp_{4n}$ of the form
$$\left(%
\begin{array}{ccccccc}
0 & -C^{-1} \\
C & 0 \end{array}%
\right).
$$
Consider
$$
B_{2n}:=\left(%
\begin{array}{ccccccc}
0 & I_{2n} \\
-I_{2n} & 0 \end{array}%
\right),
$$
which is an element of $\Sp_{4n}$.
By translating our set $J_3P_4$ from the left by $B_{2n}$, we obtain
$$B_{2n}(J_3P_4)=
\left\{
 \left(%
\begin{array}{ccccccc}
C& 0\\
0 & C^{-1} \end{array}%
\right) \ : \
C\in \Sp_{2n} \right\}.$$
It is clear that $P_4$, as a submanifold of $\SO_{16n}$, is isometric
to $B_{2n}(J_3P_4)$, and the latter is a subspace of $\Sp_{4n}$.
More precisely, it is the image of    the embedding $\Sp_{2n} \to \Sp_{4n}$,
\begin{equation}\label{imed}C\mapsto \left(%
\begin{array}{ccccccc}
C& 0\\
0 & C^{-1} \end{array}%
\right).\end{equation}
The  metric on $\Sp_{4n}$  induced by its embedding in $\SO_{16n}$ is equal to the
standard metric multiplied by 4. (Indeed, $\Sp_{4n}$ is  contained in the subspace of
all elements of $\SO_{16n}$ which commute with $J_1$, which  is $\U_{8n}$, and the
resulting embedding $\Sp_{4n}\subset \U_{8n}$ is just the one described at the beginning of
Section \ref{inclubot}; moreover, the Riemannian metric on $\U_{8n}$ induced by its embedding in $\SO_{16n}$ is twice its standard metric.)  Thus, the submanifold metric induced on
$\Sp_{2n}$ via the embedding (\ref{imed}) is the standard one multiplied by  8. Finally note that, from the previous considerations, $\Sp_{2n}$ equipped with this metric is isometric to the subspace $P_4$ of $\SO_{16n}$.

We will now prove point (b) of Proposition \ref{isotype}  for a particular choice of
$ J_5, J_6$, and $J_7$.
As it has been pointed out by Eschenburg in
\cite[Section 19]{Es},   we may assume that $P_4=\Sp_{2n}$ and
$P_5, P_6, P_7, P_8$ are subspaces of $\Sp_{2n}$ defined
as follows: first, $P_5:=\{J' \in \Sp_{2n} \ : \ (J')^2=-I\}$;
then, for $\ell =5,6$, or $7$, pick $J'_\ell\in P_\ell$ and define
$P_{\ell +1}$ as one of the
top-dimensional  components of the space $\{J' \in P_\ell \ : \ J'J'_\ell+J'_\ell J'=0\}$.
In what follows $\Sp_{2n}$ is regarded as the space of all $\bR$-linear endomorphisms 
 of $\bH^{2n}$ which preserve the norm of a vector and commute with 
 $R'_i$ and $R'_j$, the operators given by multiplication from the right by
 $i$, respectively $j$.  
 
 We first consider $J'_5:\bH^{2n}\to \bH^{2n}$ given by multiplication from the left by
 the negative of the quaternionic unit $i$,
 that is
  $$J'_5(q_1, \ldots, q_{2n}) := -i(q_1, \ldots, q_{2n}),$$
 for all $(q_1, \ldots, q_{2n})\in \bH^{2n}$. One can see that $J'_5$ is an element of 
 $\Sp_{2n}$ and satisfies $(J'_5)^2=-I$.
 Note that the $1$-eigenspace of $R'_iJ'_5$ is $\bC^{2n}$, which is canonically
 embedded in $\bH^{2n}$.
Next, we take $J'_6: \bH^{2n}\to \bH^{2n}$ given by multiplication from
the left by $-j$:
 $$J'_6(q_1, \ldots, q_{2n}) := -j(q_1, \ldots, q_{2n}),$$
 for all $(q_1, \ldots, q_{2n})\in \bH^{2n}$. As before, $J'_6$ is in
 $\Sp_{2n}$ and $(J'_6)^2=I$. We also have $J'_5J'_6=-J'_6J'_5$.
The composed map $R'_jJ'_6$ leaves $\bC^{2n}$ invariant
and the  1-eigenspace of $R'_jJ'_6|_{\bC^{2n}}$ is $\bR^{2n}$,
which is canonically embedded in $\bC^{2n}$.
Finally, we choose $J'_7$ to be the map $\bH^{2n}\to \bH^{2n}$,
$$J'_7(q_1, \ldots, q_{2n}) := -k(q_1, \ldots,q_n, -q_{n+1},\ldots, -q_{2n}),$$
 for all $(q_1, \ldots, q_{2n})\in \bH^{2n}$. This new map is in
 $\Sp_{2n}$, it squares  to $-I$, and it anticommutes with both
 $J'_5$ and $J'_6$. The composed map $R'_kJ'_7$ leaves $\bR^{2n}$ 
 invariant and we have
 $$R'_kJ'_7(x_1, \ldots, x_{2n}) = (x_1, \ldots, x_n, -x_{n+1}, \ldots, -x_{2n})$$
 for all $(x_1, \ldots, x_{2n})\in \bR^{2n}$. Consequently, the 1-eigenspace of $R'_kJ'_7|_{\bR^{2n}}$ is $\bR^n$, that is,  the subspace of
$\bR^{2n}$ consisting of all vectors with the last $n$ components equal to 0.
The $(-1)$-eigenspace of $R'_kJ'_7|_{\bR^{2n}}$ is $(\bR^{n})^\perp$,
the orthogonal complement of $\bR^n$ in $\bR^{2n}$.

We are especially interested in the embedding of $P_8$ in $\Sp_{2n}$.
By \cite[p.~141]{Mi}  (see also \cite[Section 19, item 8']{Es}), one can identify $P_8$  with one of the two
connected components of the space of all orthogonal transformations from
$\bR^n$ to $(\bR^n)^\perp$; the identification is given by
$J'\mapsto J'_7J'|_{\bR^n}$. We deduce that
$J'_7P_8$ is one of the two connected components of the subspace of $\Sp_{2n}$
consisting of all matrices of the form
$$\left(%
\begin{array}{ccccccc}
0 & -D^{-1} \\
D & 0 \end{array}%
\right)$$
where $D\in \O_n$. We may assume that
$J'_7P_8$ is the space of all matrices of the form above with
$D\in \SO_{n}$.
Let us now consider the matrix
$$
B_{n}:=\left(%
\begin{array}{ccccccc}
0 & I_{n} \\
-I_{n} & 0 \end{array}%
\right)
$$
and observe that
$$B_n(J'_7P_8)=\left\{\left( \begin{array}{ccccccc}
D & 0 \\
0 & D^{-1} \end{array}%
\right)
\ : \ D\in \SO_n\right\}.$$
The subspaces $P_8$ and $B_n(J'_7P_8)$ of $\Sp_{2n}$ are isometric.
We only need to characterize the submanifold metric on $\SO_n$
induced by the embedding $\SO_n\to \Sp_{2n}$,
$$D\mapsto
\left( \begin{array}{ccccccc}
D & 0 \\
0 & D^{-1} \end{array}%
\right),$$
where $\Sp_{2n}$ is equipped with the standard metric multiplied by eight (by point (a)).
To this end we first look at the subspaces $\U_{2n}$ and $\SO_{2n}$ of
$\Sp_{2n}$: the metric induced on $\U_{2n}$ is eight times its canonical metric
(see the beginning of Section \ref{inclubot}),
thus also the metric on $\SO_{2n}$ is eight times its canonical metric.
Consequently, the metric on $\SO_n$ we are interested in is equal to the standard one multiplied by 16.

If $J_1, \ldots, J_7$ are now arbitrary, then the
 results stated by Proposition \ref{isotype}
remain true. Indeed,  one can easily see that in this general
set-up, the spaces $P_0$ and $P_1$ are the same as above,
whereas each of   $P_2, \ldots, P_8$ differ from the ones
 described above by group conjugation inside $\SO_{16n}$.
\end{proof}

\section{Simple Lie groups as symmetric spaces and their involutions}\label{ligrp}

Let $G$ be a compact, connected, and simply connected simple Lie group.
Equipped with  a bi-invariant metric,
$G$ becomes a
 Riemannian symmetric space, as explained in
Example \ref{anyc}. The following result has been proved in \cite{Le}
(see the proof of Theorem 3.3 in that paper).
Since it plays an essential role in our Subsection \ref{bottch},  we decided to state it separately and  give the details of  the proof.

\begin{proposition}\label{gisom} Let $\tau : G\to G$ be an isometric involution with the property that
$\tau(e)=e$, where $e$ is the identity element of $G$. Then there exists an involutive group
automorphism $\mu:G\to G$ such that either
\begin{equation*}\tau(g)=\mu(g) \ {\it for \ all \ } g\in G \end{equation*}
or
\begin{equation*}\tau(g)=\mu(g)^{-1} \ {\it for \ all \ } g\in G. \end{equation*}
Moreover, in the second case, the space $(G^\tau)_e$ (the connected component through
$e$ of the fixed point set $G^\tau$) is a totally geodesic submanifold of $G$ which is
isometric to $G/G^\mu$.
Here $G/G^\mu$ is equipped with the symmetric space structure induced by
some bi-invariant metric on $G$.
\end{proposition}

\begin{proof} Let $\hat{G}$ be the identity component of the isometry group of $G$.
Then $\tau$ induces the involutive group automorphism $\hat{\tau} : \hat{G}\to \hat{G}$,
$f\mapsto {\tau}\circ f \circ {\tau}$. Let $\hat{\g}$ be the Lie algebra of $\hat{G}$ and
denote by $\hat{\tau}_*:\hat{\g}\to \hat{\g}$
the differential map of $\hat{\tau}$ at the point $\hat{e}$, which is the identity element of
$\hat{G}$.  We know that $\hat{G}=(G\times G)/\Delta(Z(G))$, where $Z(G)$ is the
center of $G$. Thus, if we denote the Lie algebra of $G$ by $\g$, then we have
$\hat{\g}=\g\times \g$. Consider the map $\sigma : G\times G \to G\times G$,
$\sigma(g_1,g_2)=(g_2,g_1)$, for all $g_1,g_2\in G$ along with its differential map at the identity element,
that is $\sigma_*:=(d\sigma)_e : \hat{\g}\to \hat{\g}$,
$\sigma_*(x_1,x_2)=(x_2,x_1)$, for all $x_1,x_2\in \g$.

\noindent {\it Claim 1.} $\hat{\tau}_*\circ \sigma_*=\sigma_* \circ \hat{\tau}_*$.

Indeed,  $\sigma_*$ can also be described as the differential at $e$ of the  map $\hat{G}\to \hat{G}$,
$f\mapsto s_e\circ f\circ s_e$, where  $s_e$ is the geodesic symmetry of $G$ at $e$
(see \cite[Ch.~IV, Theorem 3.3]{He}). We only need to notice that the automorphism
$\hat{G}\to \hat{G}$ described
above commutes with $\hat{\tau}$. In turn, this follows from the fact that $\tau\circ s_e=s_e\circ \tau$ (both sides of the equation are isometries of $G$ whose value at $e$ is $e$ and whose differential map at $e$ is equal to $-(d\tau)_e$).

Let us now observe that $\hat{\tau}_*(\g\times \{0\})$ is an ideal of $\g\times \g$.
It can only be equal to $\g\times \{0\}$ or to $\{0\}\times \g$, since $\g$ is a simple
Lie algebra.

\noindent{\it Case 1.} $\hat{\tau}_*(\g\times \{0\})=\g\times \{0\}$.
There exists $\mu : \g \to \g$ an involutive Lie algebra automorphism such that
$\hat{\tau}_*(x,0)=(\mu(x),0)$, for all $x\in \g$. From Claim 1 we deduce that
$\hat{\tau}_*(0,x)=(0,\mu(x))$, for all $x\in \g$, thus
$$\hat{\tau}_*(x_1,x_2)=(\mu(x_1),\mu(x_2)),$$
for all $x_1,x_2\in \g$.  We consider the group automorphism of $G$ whose differential at
$e$ is $\mu$ and denote it also by $\mu$. We have
$$\hat{\tau}([g_1,g_2])=[\mu(g_1),\mu(g_2)],$$
for all $g_1,g_2\in G$, where the brackets $[ \ , \ ]$ indicate the coset modulo
$\Delta(Z(G))$. Using the identification $\hat{G}=(G\times G)/\Delta(Z(G))$ and the explicit form of its action on $G$ given by Equation (\ref{ligr}),
this implies
$$\tau(g_1\tau(h)g_2^{-1})=\mu(g_1)h\mu(g_2)^{-1},$$
for all $g_1,g_2,h\in G$. Thus, $\tau(g)=\mu(g)$ for all $g\in G$.

\noindent{\it Case 2.} $\hat{\tau}_*(\g\times \{0\})= \{0\}\times \g$.
This time, there exists $\mu : \g \to \g$ an involutive Lie algebra automorphism such that
$\hat{\tau}_*(x,0)=(0,\mu(x))$, for all $x\in \g$. From Claim 1 we deduce that
$\hat{\tau}_*(0,x)=(\mu(x),0)$, for all $x\in \g$, thus
$$\hat{\tau}_*(x_1,x_2)=(\mu(x_2),\mu(x_1)),$$
for all $x_1,x_2\in \g$. Again, we consider the group automorphism $\mu :G \to G$ whose differential at
$e$ is $\mu$. This time we have
$$\hat{\tau}([g_1,g_2])=[\mu(g_2),\mu(g_1)],$$ which implies
$$\tau(g_1\tau(h)g_2^{-1})=\mu(g_2)h\mu(g_1)^{-1},$$
for all $g_1,g_2,h\in G$. This implies $\tau(g)=\mu(g)^{-1}$, for all $g\in G$.

We now prove the last assertion in the proposition. We are in Case 2.
Consider the action of $G$ on $G^\tau$ given by
$G\times G^\tau \to G^\tau$, $g.x:=gx\tau(g)$, for all $g\in G$ and $x\in G^\tau$.
Since $G$ is connected, it leaves $(G^\tau)_e$ invariant.
The corresponding  action is isometric, where  $(G^\tau)_e$ is equipped with the submanifold metric induced by its embedding in $G$.

  \noindent {\it Claim 2.} The action $G\times (G^\tau)_e \to (G^\tau)_e$ is transitive.

 To justify this, we take $x\in (G^\tau)_e$ and show that there exists $g\in G$
 with $x=g.e=g\tau(g)$. Indeed, let $\gamma : \bR\to (G^\tau)_e$ be a geodesic in
 $(G^\tau)_e$  with $\gamma(0)=e$ and $\gamma(1)=x$. Since $(G^\tau)_e$ is a totally geodesic subspace of $G$, $\gamma$ is a geodesic in $G$, hence
 we have
 $\gamma(t) =\exp(tX)$, for all $t\in \bR$, where $X$ is an element of $\g$.
 From $\tau(\gamma(t)) = \gamma(t)$ for all $t\in \bR$ we deduce
 that $(d\tau)_e(X)=X$. Then $g:=\exp(\frac{1}{2}X)$ is in $(G^\tau)_e$ and  we have
 $g\tau(g)=g^2=\exp(X)=x$.

It only remains to observe that the stabilizer of $e$ under the action mentioned in the claim is
 $G^\mu$. \end{proof}

\bibliographystyle{abbr}

\end{document}